%% file: main.tex
\documentclass[12pt]{amsart}
\usepackage[utf8]{inputenc}
\usepackage{amssymb,hyperref,mathrsfs, mathtools,cancel}
\usepackage{hyperref}
\usepackage[mathscr]{euscript}
\usepackage[lmargin=1.25in,rmargin=1.25in,tmargin=1in,bmargin=1in]{geometry}
\usepackage{array}
\usepackage{enumitem}
\usepackage{color}
\usepackage{makecell}
\usepackage{tikz,graphicx}
\usepackage{stmaryrd}
\usetikzlibrary{matrix}
\usetikzlibrary{arrows}
\usetikzlibrary{positioning}
\usepackage{overpic}
\usepackage{todonotes}
\usepackage{subcaption}
\usepackage{overpic}
\usepackage[all]{xy}
\newcommand{\M}{\mathsf{M}}
\newcommand{\R}{\mathcal{R}}
\newcommand{\Z}{\mathbb{Z}}
\newcommand{\n}{\mathbf{N}}
\newcommand{\s}{\mathbf{S}}
\newcommand{\e}{\mathbf{E}}
\newcommand{\w}{\mathbf{W}}
\renewcommand{\L}{\mathcal{L}}

\DeclareMathOperator{\id}{id}

\input{defs.tex}

\graphicspath{{figs/}}

\begin{document}

\title[Relating tangle invariants for Kh and HFK]{Relating tangle invariants for Khovanov homology and knot Floer homology}

\author{Akram Alishahi}
\thanks{AA was supported by NSF Grants DMS-1505798 and DMS-1811210.}
\thanks{ND was supported by NSF Grant DMS-1606421}
\address{Department of Mathematics, Columbia University, New York, NY 10027}
\email{\href{mailto:alishahi@math.columbia.edu}{alishahi@math.columbia.edu}}

\author{Nathan Dowlin}
\address{Department of Mathematics, Columbia University, New York, NY 10027}
\email{\href{mailto:ndowlin@math.columbia.edu }{ndowlin@math.columbia.edu}}


\date{\today}

\begin{abstract}

Ozsv\'{a}th and Szab\'{o} recently constructed an algebraically defined invariant of tangles which takes the form of a $DA$ bimodule. This invariant is expected to compute knot Floer homology. The authors have a similar construction for open braids and their plat closures which can be viewed as a filtered $DA$ bimodule over the same algebras. For a closed diagram, this invariant computes the Khovanov homology of the knot or link. We show that forgetting the filtration, our $DA$ bimodules are homotopy equivalent to a suitable version of the Ozsv\'{a}th-Szab\'{o} bimodules. In addition to giving a relationship between tangle invariants for Khovanov homology and knot Floer homology, this gives an oriented skein exact triangle for the Ozsv\'{a}th-Szab\'{o} bimodules which can be iterated to give an oriented cube of resolutions for the global construction. 

\end{abstract}

\maketitle

\tableofcontents

\section{Introduction}

Knot Floer homology and Khovanov homology are two powerful knot invariants which categorify the Alexander polynomial and the Jones polynomial, respectively. Knot Floer homology was developed by Ozsv\'{a}th-Szab\'{o} \cite{OS04:HolomorphicDisks} and independently by Rasmussen \cite{Rasmussen03:Knots}, and it is constructed as a certain Lagrangian Floer homology coming from a Heegaard diagram for the knot. Khovanov homology is a purely algebraic construction with its roots in the representation theory of the quantum group $\mathcal{U}_{q}(\mathfrak{sl}_2)$. Both Khovanov homology and knot Floer homology have constructions for tangles as well, which recover the knot invariants through gluing operations (see \cite{alishahi2018link, BarNatan05:Kh-tangle-cob, caprau2008sl, Khovanov02:Tangles, lauda2009open, Roberts:KhA, Roberts:KhD}, for tangle invariants for Khovanov homology and \cite{alishahi2016tangle, ozsvath2017bordered, ozsvath2018kauffman, petkova2016combinatorial, zibrowius2016heegaard}, for tangle invariants for knot Floer homology).  The aim of this paper is to use the authors' tangle invariant for Khovanov homology \cite{alishahi2018link} and the bimodules of Ozsv\'{a}th and Szab\'{o} for knot Floer homology \cite{ozsvath2017bordered,ozsvath2018kauffman} to give a local relationship between Khovanov homology and knot Floer homology for open braids and their plat closures. Since the invariant from \cite{alishahi2018link} only computes Khovanov homology when working with $\mathbb{Q}$ coefficients, we will work with $\mathbb{Q}$ coefficients throughout the paper. Note that the proof that Ozsv\'{a}th and Szab\'{o}'s algebraically defined invariant can be used to compute knot Floer homology has not yet been written down, but it is currently in preparation (see \cite{ozsvath2017bordered}, Section 1.3).

\subsection{Open braids} The object that Ozsv\'{a}th and Szab\'{o} assign to an open braid $b$ on $m$ strands is a $DA$ bimodule over an algebra $A$. There are several different versions of this theory; for the reader familiar with the construction, we are working with the algebras from \cite{ozsvath2018kauffman} with all strands oriented downwards, so $A=B(m,k) = B(m,k,\emptyset)$. Let $\sigma_i$ denote the elementary braid with a single positive crossing between strands $i$ and $i+1$ (again with all strands oriented downwards), and let $\sigma_{i}^{-1}$ denote the corresponding braid with a negative crossing. Let $\mathsf{OS}_{DA}(\sigma_i)$ and $\mathsf{OS}_{DA}(\sigma_i^{-1})$ denote the corresponding $DA$ bimodules.
 
We can write $b$ as a product of elementary braids ordered from the bottom of the braid to the top
\[ b = \prod_{i=1}^{N} \sigma_{j(i)} \]
with $j(i) \in \{1,-1,2,-2,...,m-1,-m+1\}$ and $\sigma_{-j}=\sigma_{j}^{-1}$. Then the $DA$ bimodule $\mathsf{OS}_{DA}(b)$ is given by the box tensor product
\[ \mathsf{OS}_{DA}(b) =  \mathsf{OS}_{DA}(\sigma_{j(1)}) \boxtimes_{A} \mathsf{OS}_{DA}(\sigma_{j(2)}) \boxtimes_{A} \cdot \cdot \cdot  \boxtimes_{A} \mathsf{OS}_{DA}(\sigma_{j(N)}) . \]

The object that the authors assign to $\sigma_i$ (resp. $\sigma_{i}^{-1}$) in \cite{alishahi2018link} is a differential bimodule $\mathsf{M}(\sigma_i)$ (resp. $\mathsf{M}(\sigma^{-1}_i)$) over an algebra $\mathcal{A}$ which is isomorphic to $A$. It is constructed as an oriented cube of resolutions. In particular, if $\mathsf{X}_{i}$ denotes the elementary singular braid with a singularization between strands $i$ and $i+1$ and $\id$ denotes the identity braid, then there is a bimodule $ \mathsf{M}(\mathsf{X}_i)$ and an identity bimodule $\M(\id)$ so that the crossing bimodules decompose as mapping cones
\[ \M(\sigma_i) = \M(\id) \xrightarrow{d^{+}} \mathsf{M}(\mathsf{X}_i) \]
\[ \M(\sigma^{-1}_i) =  \mathsf{M}(\mathsf{X}_i) \xrightarrow{d^{-}} \M(\id). \]

\noindent
The bimodule $\M(b)$ is given by the tensor product 
\[
\M(b) = \M(\sigma_{j(1)}) \otimes_{\A} \M(\sigma_{j(2)}) \otimes_{\A} \cdot \cdot \cdot \otimes_{\A} \M(\sigma_{j(N)})
.\]

In this paper, we show that the bimodules $\mathsf{M}(\sigma_i)$ and $\mathsf{M}(\sigma^{-1}_i)$ have a set of generators over which they are free left modules over the corresponding idempotent subalgebras of $\A$.
 It follows that they can be viewed as $DA$ bimodules with only $\delta_{1}^{1}$ and $\delta^1_2$ actions, where $\delta_{1}^{1}$ is given by the differential and $\delta^1_2$ describes the right multiplication. We will write these $DA$ bimodules as $\M_{DA}(\sigma_i)$ and $\M_{DA}(\sigma_{i}^{-1})$, respectively.

\begin{theorem} \label{crossinghomotopy}

Under the isomorphism $\A \cong A$, there are homotopy equivalences
\[ \M_{DA}(\sigma_i) \simeq \mathsf{OS}_{DA}(\sigma^{-1}_{i})\]
\[ \M_{DA}(\sigma_{i}^{-1}) \simeq  \mathsf{OS}_{DA}(\sigma_{i}) . \]

\end{theorem}

\begin{corollary}

For any open braid $b$, $\M_{DA}(b) \simeq \mathsf{OS}_{DA}(\overline{b})$, where $\overline{b}$ is the mirror of $b$.

\end{corollary}

This result also tells us that the Ozsv\'{a}th-Szab\'{o} bimodules can be decomposed as an oriented cube of resolutions:

\begin{corollary}

The $DA$ bimodules $\M_{DA}(\sigma_i)$ and $\M_{DA}(\sigma_{i}^{-1})$ give an oriented cube of resolutions decomposition of $\mathsf{OS}_{DA}(\sigma^{-1}_{i})$ and $\mathsf{OS}_{DA}(\sigma_{i})$, respectively.

\end{corollary}

\noindent
The homotopy equivalence from Theorem \ref{crossinghomotopy} is described in more detail at the end of the introduction. First, we will describe the analogous results for the plat maximum and minimum tangles.

\subsection{Plat closures of braids} Suppose that $b$ is an open braid on $m=2n$ strands. Let $\vee(n)$ and $\wedge(n)$ denote the plat minimum and plat maximum, respectively (see Figure \ref{maxandmins}). Then the diagram $D=\mathsf{p}(b)$ is a plat diagram for a link $L$, given by $\vee(n) \cdot b \cdot \wedge(n)$.

\begin{remark}
Note that the edges are oriented consistently with the downward-oriented braid $b$ so that there is a bivalent source at each maximum and a bivalent sink at each minimum. We are secretly viewing the $i$th bivalent vertex in the plat maximum as identified with the $i$th bivalent vertex in the plat minimum, making an oriented singularization. In Khovanov-Rozanksy $\mathfrak{sl}_2$ homology (which is isomorphic to Khovanov homology \cite{hughes2014note, KhovanovRozansky08:MatrixFactorizations}), the complex coming from the oriented singularization is quasi-isomorphic to the one coming from the unoriented smoothing, which allows us to use this trick of orienting all strands downward in a plat diagram for $D$.

\end{remark}

\begin{figure}[h!]
\scriptsize
    \centering
    \begin{subfigure}[b]{0.55\textwidth}
        \centering
\def\svgwidth{8cm} 
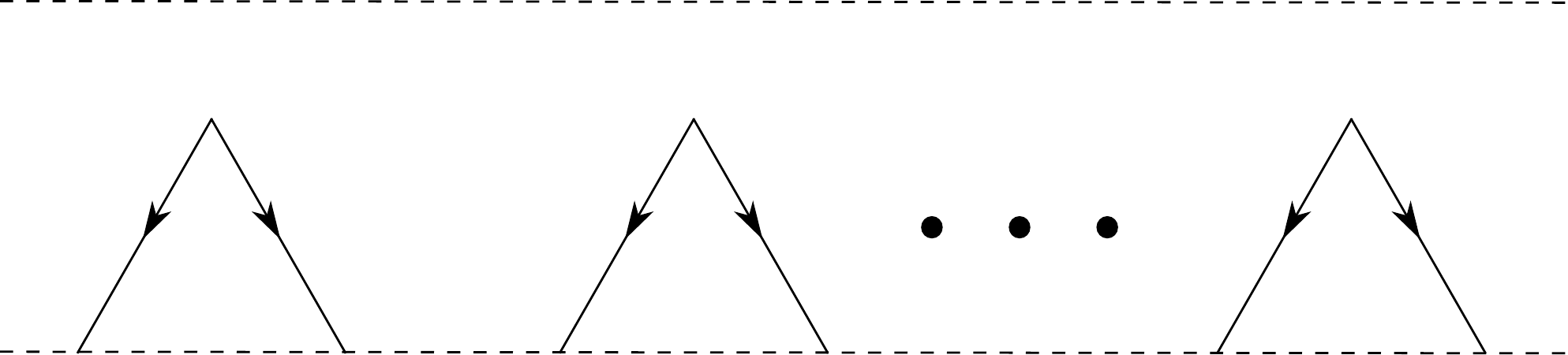
    \caption{The diagram for the plat maximum $\wedge(n)$.}
    \end{subfigure}%
    \vspace{1cm}
    \begin{subfigure}[b]{0.55\textwidth}
    \centering
\def\svgwidth{8cm} 
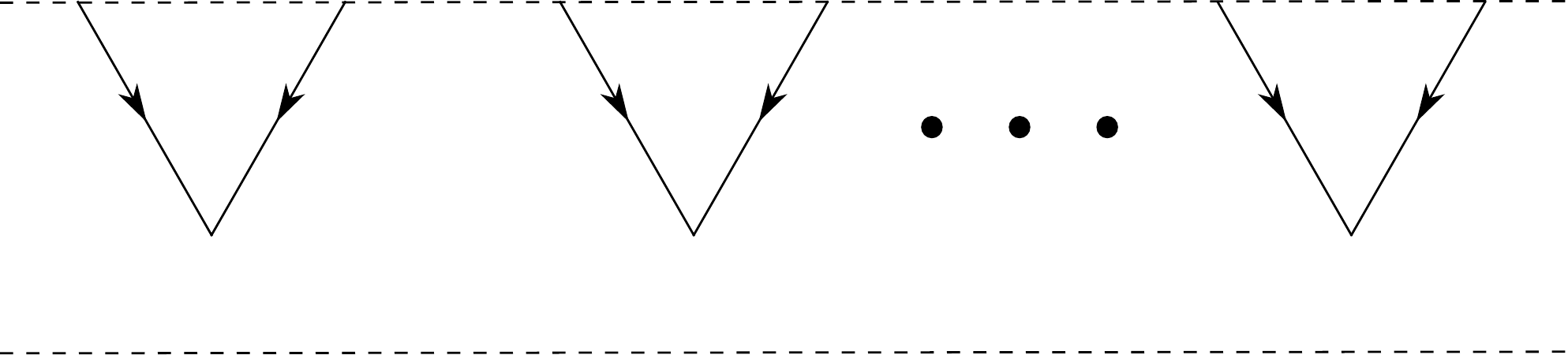
    \caption{The diagram for the plat minimum $\vee(n)$.}
    \end{subfigure}
    \caption{} \label{maxandmins}
\end{figure}

As with open braids, our bimodules $\M(\vee(n))$ and $\M(\wedge(n))$ can be viewed as a Type A structure and a Type $D$ structure, respectively, over $\A$. These will be denoted $\M_{A}(\vee(n))$ and $M_{D}(\wedge(n))$. The complex $\M(D) = \M_{A}(\vee(n)) \boxtimes_{\A} \M_{DA}(b) \boxtimes_{\A} \M(\wedge(n))$ is most naturally viewed as a curved complex with $d^2 =\omega \cdot I$. 

\begin{definition}

The total complex $C_{1 \pm 1}(D)$ is given by the tensor product 
\[ C_{1 \pm 1}(D) = \M(D) \otimes \mathsf{K}(D) \]

\noindent
where $\mathsf{K}(D)$ is a curved Koszul complex with potential $-\omega$.

\end{definition}

Note that since potentials are additive under tensor product, $d^2=0$ on $C_{1 \pm 1}(D)$. This chain complex comes with a filtration induced by height in the cube of resolutions. The following is the main result of \cite{alishahi2018link}:

\begin{theorem}[\cite{alishahi2018link}]

Let $E_{k}(D)$ denote the spectral sequence induced by the cube filtration on $C_{1 \pm 1}(D)$. Then $E_{2}(D) \cong Kh(L)$, and the total homology $E_{\infty}(D)$ is a link invariant.

\end{theorem}

\noindent
We conjectured that the total homology is $\mathit{HFK}_{2}(L)$, the deformation of knot Floer homology from \cite{dowlin2018family, dowlin2018knot}. The reduced theory $\widehat{\mathit{HFK}}_{2}(L)$ is isomorphic to the $\delta$-graded, reduced knot Floer homology of $L$.

When working over the algebra $A$, the type D module for an upper diagram is a curved type D structure, where the curvature is given in terms of which endpoints are matched in the upper diagram. In particular, the curvature is given by 
\[ (\delta^1)^2 = \sum_{\substack{ \text{matched pairs }\{i,j\} \\ \text{in upper diagram}}} U_{i}U_{j}. \]

\noindent
Ozsv\'{a}th and Szab\'{o} provide three ways of correcting this:


\begin{itemize}

\item Using a consistent orientation on $D$, append a variable $C_{i}$ to $A$ if the $i$th strand is oriented upwards satisfying $\partial(C_i) = U_i$. Then add additional differentials in terms of the $C_{i}$ so that the type D modules for an upper diagram satisfy $(\delta^1)^2=0$. This is the method used in \cite{ozsvath2018kauffman}.

\item If two strands $i$ and $j$ are matched in the upper diagram, i.e. they lie on the same component of the upper diagram, append a variable $C_{\{i,j\}}$ with $\partial(C_{\{i,j\}})=U_{i}U_{j}$. Then add additional differentials in terms of the $C_{\{i,j\}}$ so that the type D modules for an upper diagram satisfy $(\delta^1)^2=0$. This is the method used in \cite{ozsvath2017bordered}.

\item Allow each type D module for an upper diagram to have curvature, but add differentials to the minimum diagram $\vee(n)$ in terms of the matching on the upper diagram so that it is a curved type A module with curvature
\[ - \sum_{\substack{ \text{matched pairs }\{i,j\} \\ \text{in upper diagram}}} U_{i}U_{j}. \]
Since curvature is additive under box-tensor product, the complex for a closed diagram satisfies $d^2=0$. This version has not appeared in their papers, but it has been in their lecture series \cite{alfieri2018introduction}.


\end{itemize}

The third method is frequently referred to as the \emph{curved construction}. Note that in the curved construction, the curvature of both the upper diagram and the lower diagram depends on the matching of the upper diagram. In this paper, we will use a method which most closely resembles the curved construction, but treats the maxima and minima more symmetrically. In particular, the curvature for an upper diagram will depend on the matching in the upper diagram, while the curvature for a lower diagram will depend on the matching in the lower diagram. 

\begin{convention}

For the plat maximum $\mathsf{OS}_{D}(\wedge(n))$, we assign the standard curved type D structure over $A$ from Ozsv\'{a}th and Szab\'{o}'s curved construction. This construction is given by the type D structure from \cite{ozsvath2018kauffman} with no $C_{i}$ variables appearing in the algebra, as all strands are oriented downwards. 

\end{convention}

Thus, the curvature of the type D structure for any upper diagram is given by
\[\sum_{\substack{ \text{matched pairs }\{i,j\} \\ \text{in upper diagram}}} U_{i}U_{j}.\]
\noindent
This matches the curvature of $M_{D}$ of the same upper diagram.

For the minimum, we make a choice that has not yet been explored in the literature. The type D structure for the plat maximum $\wedge(n)$ can be viewed as a left module over $A$. The type A structure for the plat minimum $\vee(n)$ is obtained by simply changing this action to a right action, then viewing the module as a type A module. It turns out that this construction is equivalent to a minor modification of the type A structure from \cite{ozsvath2018kauffman}:

\begin{convention}

For the plat minimum $\mathsf{OS}_{A}(\vee(n))$, we use the construction for a minimum from the ``alternative construction" in Section 9.2 of \cite{ozsvath2018kauffman}, again with the $C_{i}$ variables set equal to zero. We also make the following change: Ozsv\'{a}th and Szab\'{o} treat the absolute minimum as special, and assign a different bimodule to it than to the other $n-1$ minima. We will treat the global minimum the same way as the local minima.

\end{convention}

The curvature of the type A structure for any lower diagram is given by
\[-\sum_{\substack{ \text{matched pairs }\{i,j\} \\ \text{in lower diagram}}} U_{i}U_{j}.\]
Again, this matches the curvature of $M_{A}$ of the same lower diagram. Thus, the total curvature for a closed diagram is given by
\[ \sum_{\substack{ \text{matched pairs }\{i,j\} \\ \text{in upper diagram}}} U_{i}U_{j} - \sum_{\substack{ \text{matched pairs }\{i,j\} \\ \text{in lower diagram}}} U_{i}U_{j}.   \]
\noindent
where $\omega$ is the curvature for our complex $M(D)$.

\begin{definition}

Let $\mathsf{OS}(D)$ be the box tensor product
\[ \mathsf{OS}(D) = \mathsf{OS}_{A}(\vee(n)) \boxtimes_{A} \mathsf{OS}_{DA}(b) \boxtimes_{A} \mathsf{OS}_{D}(\wedge(n)) \]
with $\mathsf{OS}_{A}(\vee(n))$ and $\mathsf{OS}_{D}(\wedge(n))$ defined as above. 

\end{definition}


\begin{theorem}

With these conventions, we have isomorphisms 
\[ \M_{A}(\vee(n)) \cong \mathsf{OS}_{A}(\vee(n)) \]
\[ \M_{D}(\wedge(n)) \cong \mathsf{OS}_{D}(\wedge(n)) \]
of Type A and Type D structures, respectively.

\end{theorem}

\begin{corollary}

For any plat diagram $D=\mathsf{p}(b)$, there is a homotopy equivalence of curved complexes $\M(D) \simeq \mathsf{OS}(D)$ and a homotopy equivalence of complexes $C_{1 \pm 1}(D) \simeq \mathsf{OS}(D) \otimes \mathsf{K}$.

\end{corollary}

The knot homology theory $\mathit{HFK}_{2}(K)$ is computed from a chain complex $\mathit{CFK}_{2}(K)$. For a suitable choice of Heegaard diagram $\mathcal{H}$, \[\mathit{CFK}_{2}(K) = \mathit{CFK}_{2}(\mathcal{H}) \otimes \mathsf{K}\] where $\mathit{CFK}_{2}(\mathcal{H})$ is the usual Heegaard Floer construction from $\mathcal{H}$ but with each basepoint assigned a particular coefficient depending on its location in the diagram. The complex $\mathit{CFK}_{2}(\mathcal{H})$ is a curved complex with the same potential $\omega$. We conjecture that this is the holomorphically defined complex that the theory $\mathsf{OS}(D)$ is computing:

\begin{conjecture} \label{rasconj}

The curved complexes $\mathsf{OS}(D)$ and $\mathit{CFK}_{2}(\mathcal{H})$ are homotopy equivalent.

\end{conjecture}

We hope that the forthcoming proofs by Ozsv\'{a}th and Szab\'{o} relating their algebraically defined invariant to knot Floer homology will shed light on this conjecture. Note that the identification in Conjecture \ref {rasconj} would give a second proof of Rasmussen's conjecture:

\begin{theorem}[\cite{dowlin2018spectral}]\label{SS}

For any knot $K$ in $S^3$, there is a spectral sequence from $\overline{Kh}(K)$ to $\widehat{\mathit{HFK}}(K)$, where $\overline{Kh}(K)$ is the reduced Khovanov homology of $K$ and $\widehat{\mathit{HFK}}(K)$ is the reduced knot Floer homology of $K$.

\end{theorem}

\subsection{Heegaard diagrams}

We will provide here a short description of what these algebraic objects should be computing from the Heegaard Floer perspective. The connections to these Heegaard diagrams have been proved yet, but they can help give the reader some intuition regarding the algebraic constructions.

The $DA$ bimodules that Ozsv\'{a}th and Szab\'{o} assign to a crossing are divided into four types of generators based on their idempotents. Ignoring the module actions $\delta^1_{k}$,
\[ \mathsf{OS}_{DA}(\sigma_i) = \mathbf{N} \oplus \mathbf{S} \oplus \mathbf{E} \oplus \mathbf{W} = \mathsf{OS}_{DA}(\sigma^{-1}_i)\]

\noindent
These four types of generators correspond to the local intersection points in the Kauffman states Heegaard diagram for a crossing (see Figure \ref{KauffmanCorners}).

\begin{figure}[h!]
\scriptsize
    \centering
    \begin{subfigure}[b]{0.45\textwidth}
        \centering
\def\svgwidth{6cm} 
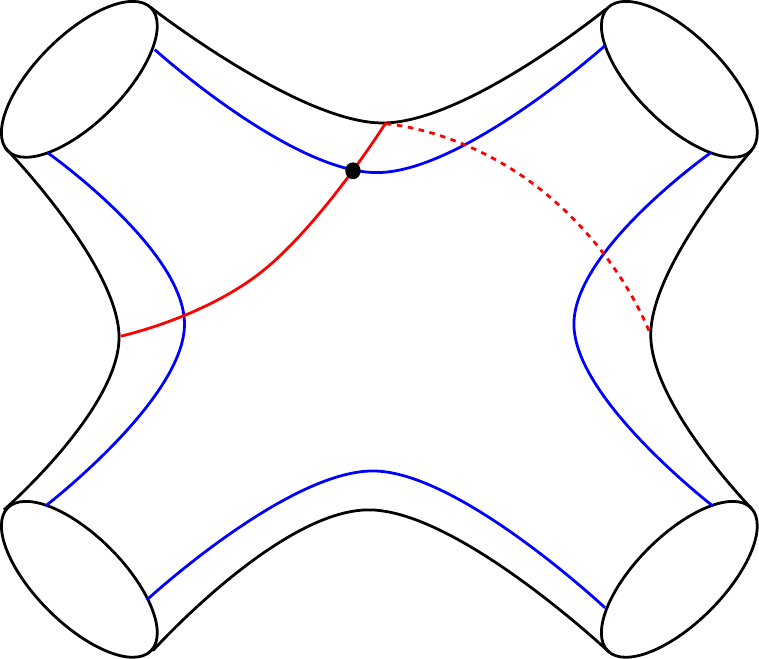
    \end{subfigure}%
    \hspace{1cm}
    \begin{subfigure}[b]{0.45\textwidth}
    \centering
\def\svgwidth{6cm} 
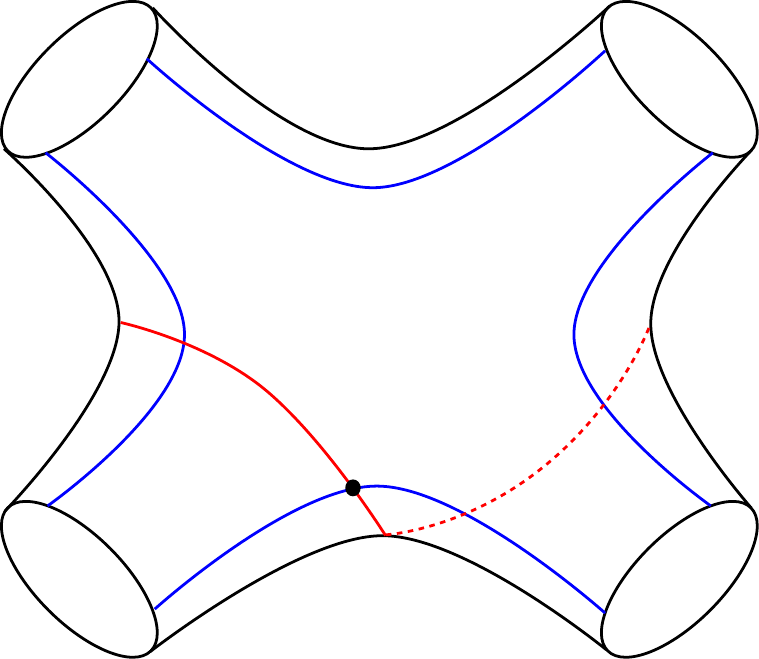

    \end{subfigure}
    \caption{The Kauffman states Heegaard diagram for a positive crossing (left) and a negative crossing (right).} \label{KauffmanCorners}
\end{figure}

The $DA$ bimodule $\mathsf{M}_{DA}(\mathsf{X}_i)$ is given by 
\[ \mathsf{M}_{DA}(\mathsf{X}_i) = N_{+} \oplus N_{-} \oplus E \oplus W \]

\noindent
with $N_{+} \cong N_{-} \cong \mathbf{N}$, $E \cong \mathbf{E}$, and $W \cong \mathbf{W}$. The identity $DA$ bimodule can be decomposed as 
\[\M_{DA}(\id) = N_{0} \oplus S \]
with $N_{0} \cong \mathbf{N}$ and $S \cong \mathbf{S}$. As a heuristic, we expect that the $DA$ bimodule $\M_{DA}(\mathsf{X}_{i})$ should come from the immersed ``figure-eight curve" and the $DA$ bimodule $\M_{DA}(\id)$ should come from the usual curve for the oriented smoothing (see Figure \ref{singandsmoothing}). This heuristic is quite similar to the oriented skein exact triangle from \cite{zibrowius2016heegaard}, where the same curves are used to give an oriented skein exact triangle in the context of 4-ended tangles.

\begin{remark}

Studying Heegaard diagrams coming from immersed $\alpha$ curves requires working with immersed Lagrangians, which only makes sense when one can show that the bounding cochains have trivial contribution \cite{akaho2008immersed}.

\end{remark}

Manion has also given a $DA$ bimodule for a singularization \cite{manion2019singular}. It is based on a different Heegaard diagram, but it is meant to be computing the same quantity; it would be interesting to see if his construction is homotopy equivalent to $\M_{DA}(\mathsf{X}_i)$.

\begin{figure}[h!]
\scriptsize
    \centering
    \begin{subfigure}[b]{0.45\textwidth}
        \centering
\def\svgwidth{6cm} 
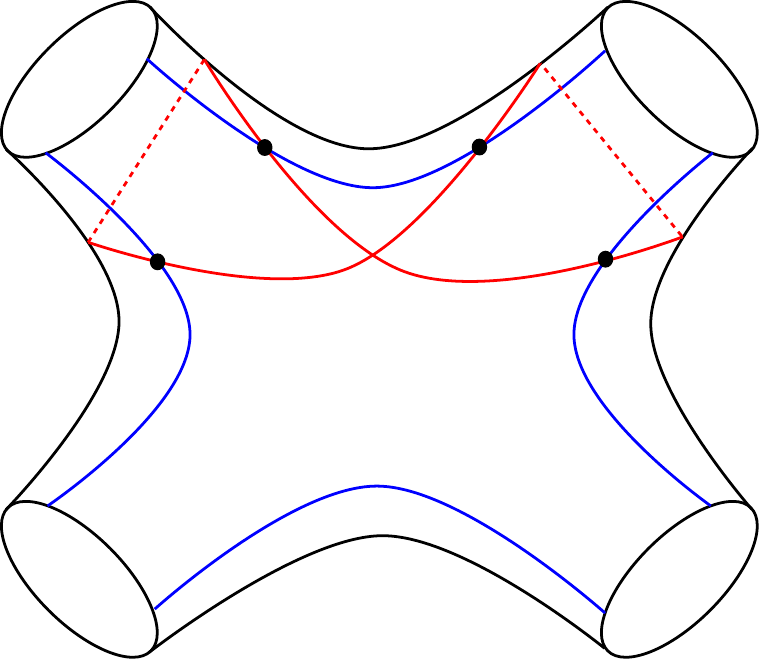
    \end{subfigure}%
    \hspace{1cm}
    \begin{subfigure}[b]{0.45\textwidth}
    \centering
\def\svgwidth{6cm} 
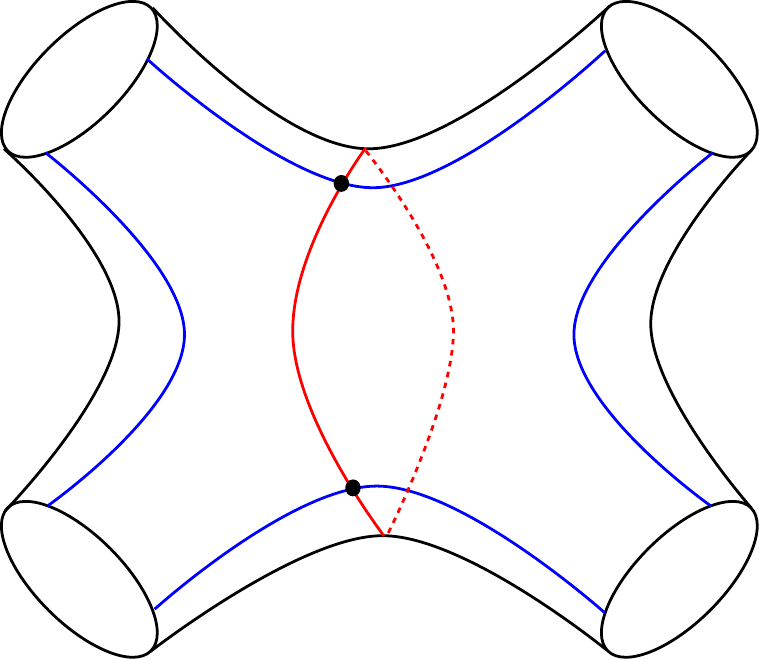

    \end{subfigure}
    \caption{The expected Heegaard diagrams for the complex $\M_{DA}(\mathsf{X}_{i})$ (left) and for the complex $\M_{DA}(\id)$ (right).} \label{singandsmoothing}
\end{figure}

Since the bimodules for crossings are each a mapping cone on $\M(\mathsf{X}_i)$ and $\M(\id)$, we have
\[ \mathsf{M}_{DA}(\sigma_i) = N_{+} \oplus N_{-} \oplus N_{0} \oplus E \oplus W \oplus S = \mathsf{M}_{DA}(\sigma^{-1}_i)\]

\noindent
In particular, it is isomorphic to the Ozsv\'{a}th-Szab\'{o} bimodules direct sum two copies of $\mathbf{N}$.

For the positive crossing, the edge map $d^{+}$ sends $N_{0}$ isomorphically to $N_{-}$. In terms of the $DA$ bimodule $\M_{DA}(\sigma_i)$, this gives an action $\delta_{1}^{1}(N_{0})=N_{-}$. Applying homological perturbation to contract this map gives a homotopy equivalent complex with generators $N_{+},S,E,W$. This $DA$ bimodule turns out to be isomorphic to the $DA$ bimodule of Ozsv\'{a}th and Szab\'{o} for a negative crossing $\mathsf{OS}_{DA}(\sigma^{-1}_i)$.

For the negative crossing, the edge map $d^{-}$ sends $N_{+}$ isomorphically to $N_{0}$. This gives $\delta^1_1(N_{+})=N_{0}$ on $\M_{DA}(\sigma^{-1}_i)$. Applying homological perturbation gives a homotopy equivalent complex with generators $N_{-}, S, E, W$. This $DA$ bimodule is isomorphic to $\mathsf{OS}_{DA}(\sigma_i)$.

The Heegaard diagrams that we expect to give rise to the two bimodules $\mathsf{M}_{DA}(\sigma_i)$ and $\mathsf{M}_{DA}(\sigma^{-1}_i)$ each come from resolving one of the intersection points between the $\alpha$ curve for $\M_{DA}(\mathsf{X}_{i})$ and the $\alpha$ curve for $\M(\id)$. The $\alpha$ curve is an immersed curve, but by applying isotopy it can be made into an embedded curve. See Figure \ref{SixGeneratorsHD} for the positive crossing case -- the negative crossing case is similar.

\begin{figure}[h!]
\scriptsize
    \centering
    \begin{subfigure}[b]{0.45\textwidth}
        \centering
\def\svgwidth{6cm} 
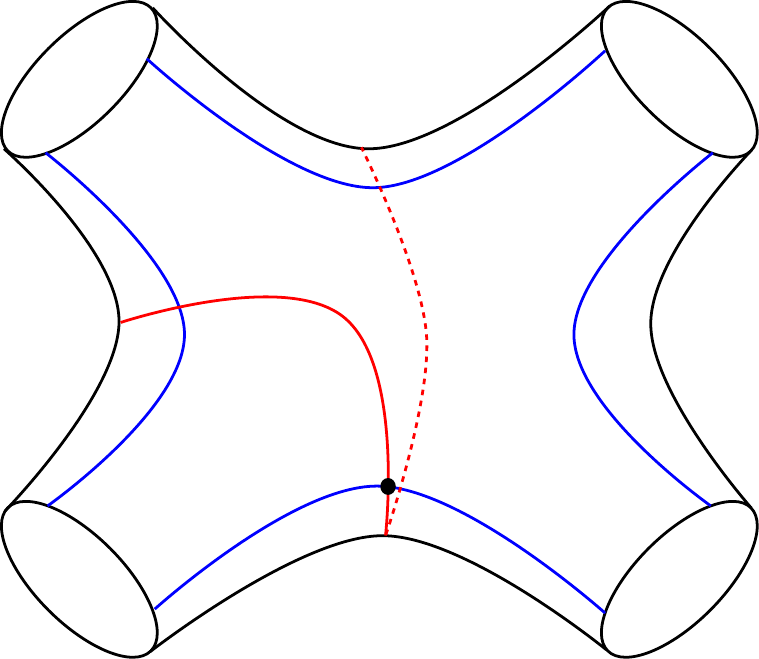
    \end{subfigure}%
    \hspace{1cm}
    \begin{subfigure}[b]{0.45\textwidth}
    \centering
\def\svgwidth{6cm} 
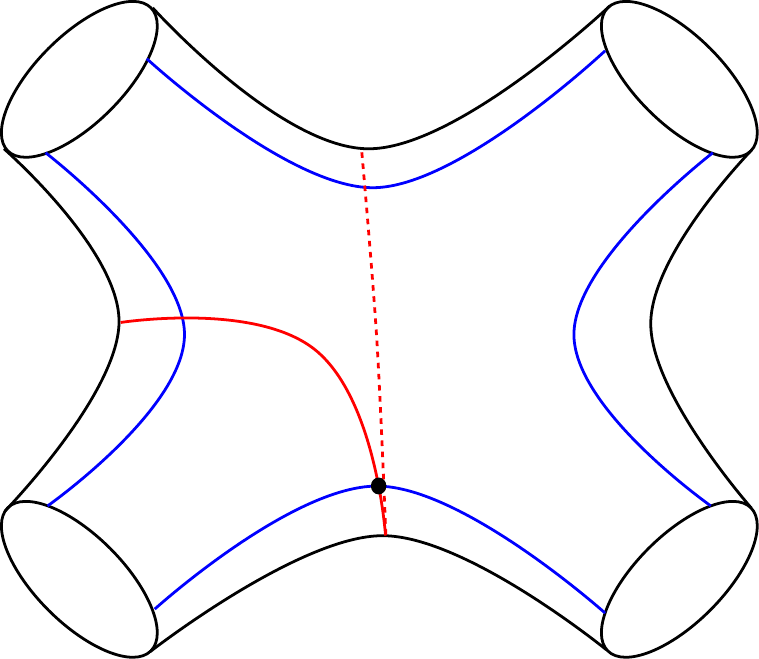

    \end{subfigure}
    \caption{The immersed (left) and embedded (right) diagrams that we expect to compute $\mathsf{M}_{DA}(\sigma_i)$. Note that they are isotopic to the Kauffman states diagram for a negative crossing.} \label{SixGeneratorsHD}
\end{figure}

\begin{remark}

Homological perturbation is somewhat easier to compute for $DD$-bimodules than for $DA$ bimodules. For this reason, we prove the above homotopy equivalences by first box-tensoring both $\M_{DA}(\sigma_i)$ and $\mathsf{OS}_{DA}(\sigma_{i}^{-1})$ with the canonical $DD$-bimodule, then using homological perturbation to show that the resulting $DD$-bimodules are homotopy equivalent. The same applies to $\M_{DA}(\sigma^{-1}_i)$ and $\mathsf{OS}_{DA}(\sigma_{i})$.

\end{remark}

\subsection{Organization} The body of the paper has three sections. In Section 2, we give background on the Ozsv\'{a}th-Szab\'{o} bimodules and describe the versions we use. In Section 3, we give background on our bimodules from \cite{alishahi2018link} and describe a basis with respect to which they can be viewed as $DA$ bimodules. In Section 4, we prove the homotopy equivalences between the two theories. We will assume that the reader is familiar with Type A and Type D structures; the relevant background is provided in \cite{ozsvath2017bordered} and \cite{ozsvath2018kauffman}.

\subsection*{Acknowledgments} We would like to thank Robert Lipshitz, Andy Manion, Peter Ozsv\'{a}th, Zolt\'{a}n Szab\'{o}, Ian Zemke, and Claudius Zibrowius for helpful discussions.

%
%
%

\section{Background I: The Ozsv\'{a}th-Szab\'{o} bimodules}

In this section, we will give background on the knot invariants of Ozsv\'{a}th-Szab\'{o} from \cite{ozsvath2018kauffman} and \cite{ozsvath2017bordered}.

\subsection{The algebras $B_{0}(m,k)$ and $B(m,k)$.}

Fix integers $k,m$. The algebra $B_{0}(m,k)$ is a certain type of strands algebra consisting of $k$ strands on $m$ points. The idempotent states are given by $k$-element subsets $S$ of $\{1,...,m\}$. Given two idempotent states $\mathbf{x},\mathbf{y}$, there is an isomorphism
\[  \phi^{\mathbf{x},\mathbf{y}}: \Z[U_{1},...,U_{m}] \to \mathbf{I_x} \cdot B_{0}(m,k) \cdot \mathbf{I_y}    \]

\noindent
As a $\Z$-module, $B_{0}(m,k)$ is the direct sum
\[  B_{0}(m,k) = \bigoplus_{\mathbf{x}, \mathbf{y}} \mathbf{I_x} \cdot B_{0}(m,k) \cdot \mathbf{I_y} .  \]

\noindent
An algebra element $b \in B_{0}(m,k)$ is called \emph{pure} if it is the image of a monomial \[b = \phi^{\mathbf{x}, \mathbf{y}}(U_{1}^{t_{1}} \cdot \cdot \cdot U_{m}^{t_{m}}). \]

The composition rules are described in terms of weights. Given an idempotent state $\mathbf{x}$, the weight vector $v^{\mathbf{x}} \in \Z^m$ is given by 
\[v_{i}^{\mathbf{x}} = \#\{x \in \mathbf{x} \mid x \ge i \} . \]

\begin{definition}
The \emph{minimal weight vector} $w^{\mathbf{x}, \mathbf{y}}$ is given by 
\[ w_{i}^{\mathbf{x},\mathbf{y}} = \frac{1}{2} |v_{i}^{\mathbf{x}} - v_{i}^{\mathbf{y}}| .\]
The weight of a pure algebra element is 
\[ w( \phi^{\mathbf{x}, \mathbf{y}}(U_{1}^{t_{1}} \cdot \cdot \cdot U_{m}^{t_{m}})) = w^{\mathbf{x}, \mathbf{y}} +(t_{1},...,t_{m}) .\]
\end{definition}

Multiplication in $B_{0}(m,k)$ is defined to be the unique $\Z[U_{1},...,U_{m}]$-equivariant weight-preserving map 
\[ (\mathbf{I_x} \cdot B_{0}(m,k) \cdot \mathbf{I_{y}}) \ast   (\mathbf{I_{y}} \cdot B_{0}(m,k) \cdot \mathbf{I_{z}}) \to  \mathbf{I_{x}} \cdot B_{0}(m,k) \cdot \mathbf{I_{z}}  . \]

\begin{remark}

Our convention differs slightly from the standard definitions in \cite{ozsvath2017bordered} and \cite{ozsvath2018kauffman}. First, we have a shift in the idempotent states so that the first slot (the one to the left of the first strand) is labeled 1 instead of $0$. Second, we truncate the algebra on the right, so that the slot to the right of the last strand is blocked (see Figure \ref{IDExAD}). In Section 12 of \cite{ozsvath2018kauffman}, Ozsv\'{a}th and Szab\'{o} show that one can truncate on the left, on the right, or even on both sides, and obtain the same knot invariant after closing off. Andy Manion refers to the once-truncated algebra as ``the Goldilocks" algebra, as a diagram on $m$ strands has $2^m$ idempotents, which is the nicest possible choice from a higher representation theory perspective.

\end{remark}

 \begin{figure}[ht]
\centering
\def\svgwidth{12cm} 
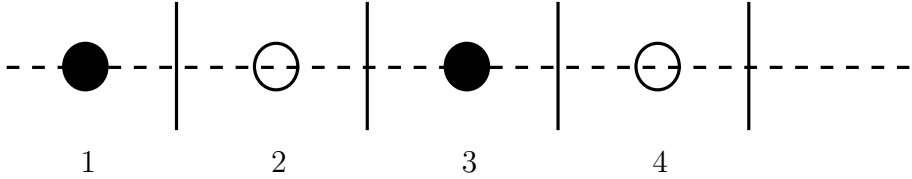
\caption{The idempotent $\{1,3\} \in B_{0}(4,2)$ in the conventions of this paper.}\label{IDExAD}
\end{figure}

 \begin{figure}[ht]
\centering
\def\svgwidth{12cm} 
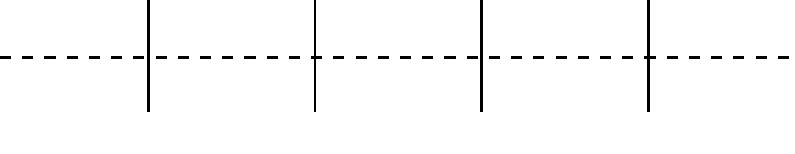
\caption{The idempotent $\{1,3\} \in B_{0}(4,2)$ in the conventions of \cite{ozsvath2018kauffman} and \cite{ozsvath2017bordered}.} \label{IDExOS}
\end{figure}

The algebra $B_{0}(m,k)$ can also be described in terms of generators $L_{i}$, $R_{i}$, for $i=1,...,m-1$. Let $\mathbf{x}$ be an idempotent state such that $i \in \mathbf{x}$, $i+1 \notin \mathbf{x}$, and let $\mathbf{y} = \mathbf{x} \setminus \{i\} \cup \{i+1\}$. Then $R_{i}^{\mathbf{x}}$ and $L_{i}^{\mathbf{y}}$ are the minimal weight elements connecting $\mathbf{x}$ to $\mathbf{y}$ and $\mathbf{y}$ to $\mathbf{x}$, respectively:
\[ R_{i}^{\mathbf{x}} = \phi^{\mathbf{x}, \mathbf{y}}(1) \hspace{3cm}  L_{i}^{\mathbf{y}} = \phi^{\mathbf{y}, \mathbf{x}}(1)  \]

\noindent
The generators $R_{i}$ (resp. $L_{i}$) are the sums of the $R_{i}^{\mathbf{x}}$ and (resp. $L_{i}^{\mathbf{y}}$) over all possible generators $\mathbf{x}$ (resp. $\mathbf{y}$):

\[   R_{i} = \sum_{  \{\mathbf{x} \mid i \in \mathbf{x}, i+1 \notin \mathbf{x} \} } R_{i}^{\mathbf{x}}   \hspace{2.5cm} L_{i} = \sum_{  \{\mathbf{y} \mid i \notin \mathbf{y}, i+1 \in \mathbf{y} \} } L_{i}^{\mathbf{y}}    \]

These generators commute with one another, provided their indices differ by at least two. Let $i,j$, with $|i-j| \ge 2$. Then 
\[ R_{i}R_{j}=R_{j}R_{i} \hspace{2cm} L_{i}L_{j} = L_{j}L_{i} \hspace{2cm} R_{i}L_{j} = L_{j}R_{i}  \]

Let $\mathbf{I_{x(i)}}$ denote the sum of all idempotents $\mathbf{I_{x}}$ such that $i \in \mathbf{x}$. Then $B_{0}(m,k)$ is generated over the idempotents by $R_{i}, L_{i}$, and $U_{i}$, modulo the relations
\[ R_{i} \cdot L_{i} = \mathbf{I_{x(i)}} \cdot U_{i} \cdot \mathbf{I_{x(i)}} \]
\[ L_{i} \cdot R_{i} = \mathbf{I_{x(i+1)}} \cdot U_{i} \cdot \mathbf{I_{x(i+1)}} \]
\noindent
with the $U_{i}$ central in the algebra. 
 
\begin{definition}

The algebra $B(m,k)$ is the quotient of $B_{0}(m,k)$ modulo the relations \[ R_{i}R_{i+1}=0, \hspace{3cm} L_{i+1}L_{i}=0 \]
and $U_{i}\cdot \mathbf{I_x}=0$ if $\{i, i+1 \} \cap \mathbf{x}=\emptyset$.

\end{definition}

Let $\mathbf{x}=x_{1}<x_{2}<...<x_{k}$, $ \mathbf{y}=y_{1}<y_{2}<...<y_{k}$ be two idempotent states. The states $\mathbf{x}$ and $\mathbf{y}$ are said to be \emph{close enough} if $|x_{i}-y_{i}| \le 1$ for all $i$. Otherwise, they are called \emph{too far}. Note that the factor $\mathbf{I_x} \cdot B(m,k) \cdot \mathbf{I_{y}}$ is non-trivial if and only if $\mathbf{x}$ and $\mathbf{y}$ are close enough.

\subsection{The algebras $A(n,k,M)$ and $A'(n,k,M)$}

Suppose $m=2n$, and let $M$ be a matching of the $2n$ points. The algebra $B(2n,k)$ can be extended to a differential algebra $A(n,k,M)$ by appending a variable $C_{ \{i,j \} }$ to $B(2n,k)$ for each matched pair $\{i,j\} \in M$. These variables are central up to signs, and they satisfy 
\[ C_{\{ i,j\} }^2 =0 \hspace{3cm} \partial(C_{\{ i,j\}})=U_{i}U_{j} \]

\noindent
The signs in the commutation relations are given by
\[ C_{\{ i,j\}} \cdot b =   b \cdot C_{\{ i,j\}} \]\[ C_{\{ i_1,j_1\}} \cdot C_{\{ i_2,j_2\}} = - C_{\{ i_2,j_2\}} \cdot C_{\{ i_1,j_1\} }   \]

\noindent
where $b$ is an element of $B(2n,k)$.

The variables $C_{\{ i,j\}}$ are called \emph{exterior} variables, and the \emph{exterior weight} of a homogeneous element $a$ of $A(n,k,M)$ is the number of exterior factors in $b$. Let $|a|$ denote the exterior weight of $a \mod 2$. The signs in the commutation can be written in terms of the exterior weight as follows:
\[ C_{\{ i,j\}} \cdot a = (-1)^{|a|} a \cdot C_{\{ i,j\}} \]

\noindent
The differential on $A(n,k,M)$ satisfies the Leibniz rule
\[ \partial(a \cdot b) = \partial(a) \cdot b + (-1)^{|a|}a \cdot \partial(b) .\]

There is also a dual algebra to $A(n,k,M)$ given by $A'(n,2n-k,M)$. It is obtained from $B(2n,2n-k)$ by appending variables $E_{1},...,E_{2n}$ satisfying 
\[ E_{i}^2=0 \hspace{3cm} \partial(E_{i}) = U_{i} \]

\noindent
The $E_{i}$ are also exterior variables, and the commutation relations are given by 
\[ E_{i} \cdot b = b \cdot E_{i} \]
\[ E_{i} \cdot E_{j} = - E_{j} \cdot E_{i} \text{ for } \{i,j\} \notin M . \]

\noindent
Note that for each pair $\{i,j\} \in M$, there is a non-zero commutator $\llbracket E_{i}, E_{j} \rrbracket=E_{i}E_{j} + E_{j}E_{i}$. 

The $DA$ bimodules in \cite{ozsvath2017bordered} are defined over $A(n,k,M)$, while the $DD$-bimodules are defined as left  $A(n,k,M)-A'(n,k,M)$ bimodules.

\subsection{The algebras $A$ and $A'$}

Let $\mathcal{S}$ denote a subset of $\{1,2,...,m\}$. The subset $\mathcal{S}$ records which strands in the knot diagram are pointed downwards at this particular slice. In \cite{ozsvath2018kauffman}, Ozsv\'{a}th and Szab\'{o} define the algebra $B(m,k,\mathcal{S})$ which is obtained by appending a variable $C_{i}$ to $B(m,k)$ for each $i \in \mathcal{S}$, which behaves much like the $E_{i}$ from above:
\[ C_{i}^2 = 0 \hspace{3cm} \partial(C_{i})=U_{i} \]

\noindent
The $C_{i}$ are exterior variables, and the commutation relations are given by 
\[ C_{i} \cdot b = b \cdot C_{i} \]
\[ C_{i} C_{j} =-C_{j} C_{i} \]

Let $\mathcal{S}_{1}$, $\mathcal{S}_{2}$ be subsets such that 
\[ \mathcal{S}_{2} = \{1,...,m\} \setminus \mathcal{S}_1 .\]

\noindent
Ozsv\'{a}th and Szab\'{o} show that $B(m,k,\mathcal{S}_1)$ is dual to $B(m,m-k, \mathcal{S}_{2})$. Thus, the algebra $A(n,k,M)$ can be viewed as $B(2n,k, \emptyset)$ with $C_{\{i,j\}}$ variables added for the matching, and $A'(n,2n-k,M)$ can be viewed as $B(2n,k, \{1,2,...,2n\})$ with a commutator $\llbracket E_{i}, E_{j} \rrbracket$ added for each matching.


We take the philosophy that working over $A(n,k,M)$ and $A'(n,2n-k,M)$ corresponds to working with a Heegaard diagram with all strands pointing downward. This means that the orientations don't agree at the cups and caps, which aligns with the Heegaard diagrams used for motivation in \cite{alishahi2018link}, which have two $O$ basepoints in a row at each cup and cap.


Let $A(n,k)$ denote the algebra $A(n,k,M) / \{ C_{\{i,j\}}=0\}$. Then $A(n,k) = B(2n,k,\emptyset)$, and the dual algebra is given by $A'(n,2n-k)=B(2n,2n-k,\{1,2,...,2n\})$. Moving to these algebras that forget the matching results in computing the curved version of Ozsv\'{a}th and Szab\'{o}'s theory \cite{alfieri2018introduction}. We choose to work with these `unmatched' algebras $A(n,k)$ and $A'(n,2n-k)$, as it provides the best context for comparison with Khovanov homololgy. When the context is clear, we will drop the indices and simply denote them $A$ and $A'$.


\subsection{The alternate generators $R'_{i}$ and $L'_{i}$}

When using the algebra $A'$, it will be useful to have generators $R'_{i}$ and $L'_{j}$ that anti-commute instead of commute for $|i-j| \ge 2$, as it makes the formulas for the $DD$-bimodules much cleaner.

For each idempotent state $\mathbf{x}$, define
\[ f(i, \mathbf{x}) = \sum_{x \in \mathbf{x}, \hspace{1mm} x<i-1 } x .\]

\noindent
Then setting 
\[ \mathbf{I_x} \cdot R'_i = (-1)^{f(i, \mathbf{x})} \mathbf{I_x} \cdot R_i \]
\[ \mathbf{I_x} \cdot L'_i = (-1)^{f(i, \mathbf{x})} \mathbf{I_x} \cdot L_i \]
gives a new set of generators satisfying
\[ R'_{i}\cdot R'_{j} = -R'_{j} \cdot R'_{i} \hspace{2cm} R'_{i} \cdot L'_{j} = -L'_{j} \cdot R'_{i} \hspace{2cm} L'_{i} \cdot L'_{j} = - L'_{j} \cdot L'_{i} \]
for $|i-j| \ge 2$, and 
\[ L'_{i} \cdot R'_{i} = L_{i} \cdot R_{i} \hspace{3cm} R'_{i} \cdot L'_{i} = R_{i} \cdot L_{i} .\]

\subsection{The canonical $DD$-bimodule}

Let $A = A(n,k)$, $A' = A'(n,2n-k)$. The canonical $DD$ bimodule for $2n$ downward pointing strands with no crossings is a $DD$-bimodule over $A - A'$. It will be denoted ${}^{A, A'}\mathsf{OS}(\id)$.

This bimodule is generated over $\Z$ by complementary pairs of idempotent states in $A, A'$. In particular, let $\mathbf{K}$ denote the free $\Z$-module generated by $\mathbf{K_{x}}$ over all idempotent states $\mathbf{x} \in A(n,k)$. $\mathbf{K}$ can be viewed as a left module over the idempotent subalgebras $\mathbf{I}(A) \otimes \mathbf{I}(A')$ as follows:
\[ 
(\mathbf{I_{x}} \otimes \mathbf{I_{y}}) \cdot \mathbf{K_{w}}= \begin{cases} \mathbf{K_{w}} \text{ if }\mathbf{x}=\mathbf{w} \text{ and } \mathbf{y}=\mathbf{\overline{w}} \\
0 \text{ otherwise}
\end{cases}
\]

Consider the algebra element $a \in A \otimes A'$ given by 
\[ a = \sum_{i=1}^{2n}(R_{i} \otimes L'_{i} + L_{i} \otimes R'_{i}) - \sum_{i=1}^{2n} U_{i} \otimes E_{i}    \]

\noindent
The differential on the $DD$-bimodule ${}^{A, A'}\mathsf{OS}(\id)$ is given by
\[ \delta^1: \mathbf{K} \to (A \otimes A') \otimes_{\mathbf{I}(A)\otimes\mathbf{I}(A')} \mathbf{K} \]
\[ \delta^1(x) = a \otimes x \]

\begin{lemma}[\cite{ozsvath2017bordered}]
The identity $DD$-bimodule $\mathsf{OS}_{DD}(\id)$ is quasi-invertible.
\end{lemma}

\subsection{The Ozsv\'{a}th-Szab\'{o} $DA$ bimodule for a maximum}

Let $\wedge_{c}$ denote the $(2n+2, 2n)$ tangle which has a maximum between strands $c$ and $c+1$ and is the identity on the remaining strands (see Figure \ref{maxc}). The corresponding algebras are $A_{1}=A(n,k)$, $A_{2}=A(n+1, k+1)$. Define $\phi_{c}: \{1,...,2n\} \to \{1,..,2n+2\} $ by 
\[ 
\phi_{c}(j) = \begin{cases}
j  & \text{ if }j<c+1 \\
j+2 & \text{ if }j \ge c+1 .\\
\end{cases}
\]

\begin{figure}[ht]
\centering
\def\svgwidth{12cm} \scriptsize
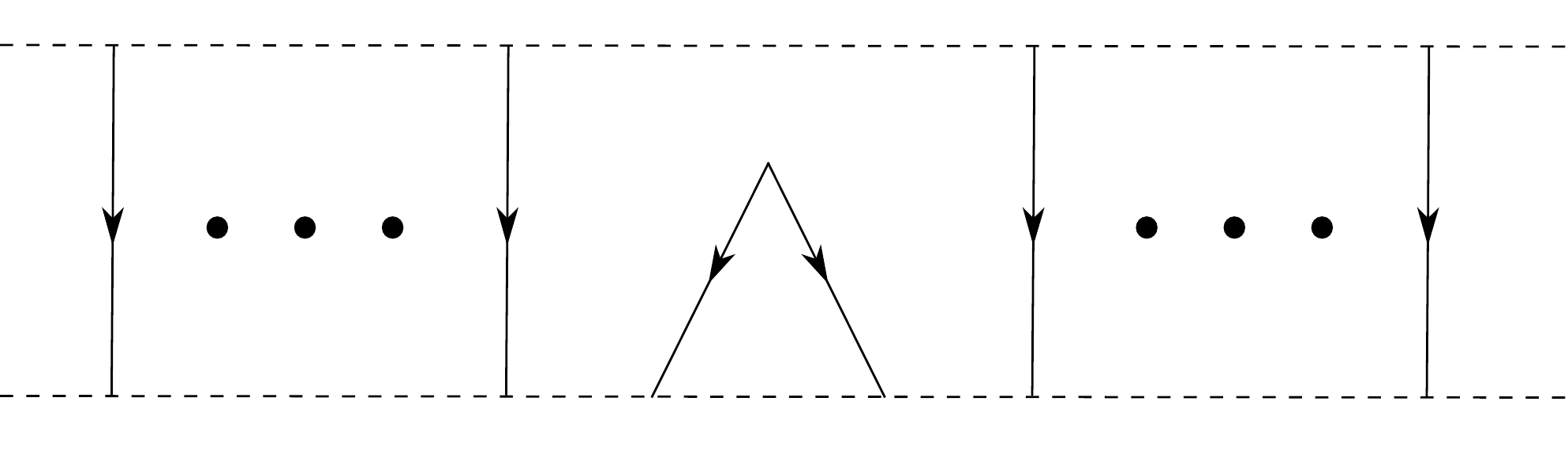
\caption{The diagram for the maximum $\wedge_{c}$.} \label{maxc}
\end{figure}

Recall that an idempotent state $\mathbf{y}\in\mathbf{I}(A_2)$ is called \emph{allowed} if 
\[c+1\in\mathbf{y}\quad\quad\text{and}\quad\quad |\mathbf{y}\cap\{c,c+2\}|\le 1.\]
Each allowed idempotent state $\mathbf{I_y} \in A_{2}$ determines a unique idempotent state $\mathbf{I_x} \in A_{1}$ by
\[
\mathbf{x} = \psi(\mathbf{y}) = \begin{cases}
\phi_{c}^{-1}(\mathbf{y}) & \text{if } c+2 \notin \mathbf{y} \\
\phi_{c}^{-1}(\mathbf{y}) \cup \{c\} & \text{if }c+2 \in \mathbf{y}\\
\end{cases}
\]

\begin{definition}

For each allowed idempotent state $\mathbf{y} \in \mathbf{I}(A_2)$, define $\mathbf{Q_y}$ to be the generator of $\mathsf{OS}_{DA}(\wedge_c)$ with left idempotent $\mathbf{y}$ and right idempotent $\psi(\mathbf{y})$
\[ \mathbf{Q_y} = \mathbf{I_y} \cdot \mathbf{ Q_{y}} \cdot \mathbf{I}_{\psi(\mathbf{y})} .\] 

\end{definition}

The generators of $\mathsf{OS}_{DA}(\wedge_c)$ decompose into three types:
\[
\begin{split}
 \mathbf{X} =& \bigoplus_{ \{ \mathbf{y} \mid \mathbf{y} \cap \{c, c+1, c+2\}\} = \{c, c+1\} } \mathbf{Q_y}  \\
 \mathbf{Y} =& \bigoplus_{ \{ \mathbf{y} \mid \mathbf{y} \cap \{c, c+1, c+2\}\} = \{c+1, c+2\} } \mathbf{Q_y}  \\
 \mathbf{Z} =& \bigoplus_{ \{ \mathbf{y} \mid \mathbf{y} \cap \{c, c+1, c+2\}\} = \{c+1\} } \mathbf{Q_y}  
\end{split}
\]

\noindent
As a $\Z$-module, $\mathsf{OS}_{DA}(\wedge_c)$ is the direct sum 
\[\mathsf{OS}_{DA}(\wedge_c) = \mathbf{X} \oplus \mathbf{Y} \oplus \mathbf{Z}  \]
The differential $\delta^1_1$ is given by 
\[ 
\begin{split}
\delta^1_1(\mathbf{X}) &= R_{c+1}R_{c} \otimes \mathbf{Y}   \\
\delta^1_1(\mathbf{Y}) &= L_{c}L_{c+1} \otimes \mathbf{X}  \\
 \delta^1_1(\mathbf{Z}) &= 0   \\
 \end{split}
 \]
 A concise description of $\delta^1_2$ requires the map $\Phi_{\mathbf{x}}$ described in \cite{ozsvath2018kauffman}, Lemma 8.1:
 \begin{lemma}
 
  If $\mathbf{x}$ is an allowed idempotent state for $A_2$ and $\mathbf{y}$ is an idempotent state for $A_1$ so that $\psi(\mathbf{x})$ and $\mathbf{y}$ are close enough, then there is an allowed idempotent state $\mathbf{z}$ with $\psi(\mathbf{z}) = \mathbf{y}$ so that there is a map
  \[
  \Phi_{\mathbf{x}}: \mathbf{I}_{\psi(\mathbf{x})} \cdot A_{1} \cdot \mathbf{I_y} \to \mathbf{I_x} \cdot A_2 \cdot \mathbf{I_z}
  \]
 with the following properties:
 \begin{itemize}
     \item $\Phi_{\mathbf{x}}$ maps the portion of $\mathbf{I}_{\psi(\mathbf{x})} \cdot A_{1} \cdot \mathbf{I_y}$ with weights $(v_{1},...,v_{2n})$ surjectively onto the portion of $\mathbf{I_x} \cdot A_2 \cdot \mathbf{I_z}$ with weights $w_{\phi_c(i)} = v_i$ and $w_{c+1}=w_{c+2}=0$.
     \item For any $a$ in $\mathbf{I}_{\psi(\mathbf{x})} \cdot A_{1} \cdot \mathbf{I_y}$, $\Phi_{\mathbf{x}}(U_{i} \cdot a) = U_{\phi_{c}(i)} \cdot \Phi_{\mathbf{x}}( a)$.
 \end{itemize}
 
 \noindent
 Moreover the state $\mathbf{z}$ is uniquely determined by the existence of such a $\Phi_{\mathbf{x}}$.
 \end{lemma}
 
  The multiplication map $\delta^1_2$ essentially consists of pushing the multiplication in $A_1$ to multiplication in $A_2$ by $\Phi_{\mathbf{x}}$. In particular, if $\mathbf{x}$ is an allowed idempotent state for $A_2$ and $a=\mathbf{I}_{\psi(\mathbf{x})} \cdot a \cdot \mathbf{I_y}$ is a non-zero algebra element in $A_1$, then 
  \begin{equation} \label{capmultiplication}
  \delta^1_2(\mathbf{Q_x}, a) = \Phi_{\mathbf{x}}(a) \otimes \mathbf{Q_z}
 \end{equation}
  There are no higher multiplication maps $\delta^1_i$ for $i>2$.
  
 \subsection{The Ozsv\'{a}th-Szab\'{o} type D module for the plat maximum $\wedge(n)$}
 
 Consider the plat maximum diagram consisting of $n$ caps in Figure \ref{NCaps}. We will denote this diagram by $\wedge(n)$. The type D module that Ozsv\'{a}th and Szab\'{o} assign to $\wedge(n)$ is defined over $A=A(n,n)$.

 \begin{figure}[ht]
\centering
\def\svgwidth{12cm} \scriptsize
\input{figs/NCaps.pdf_tex}
\caption{The diagram $\wedge(n)$ for $n$ downward-oriented caps.}\label{NCaps}
\end{figure}
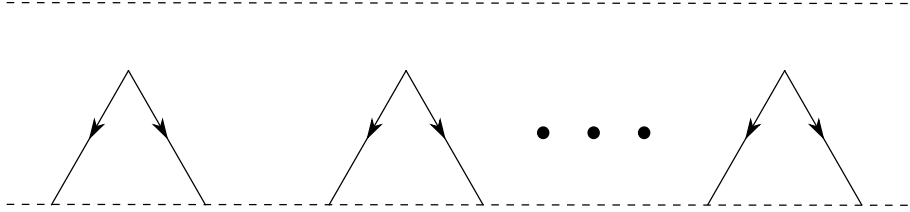

Suppose that we order the maxima from left to right so that the left maximum is lowest in the diagram and the right maximum is highest. Then at each stage, the maximum is leftmost in the diagram, so it is given by $\wedge_1$ and the corresponding bimodule is $\mathsf{OS}_{DA}(\wedge_1)$. With this convention, we can write 
\[ \wedge(n) = \underbrace{\wedge_1 \ast \wedge_1 \ast ... \ast \wedge_1}_{n \text{ times } }\]
\[ \mathsf{OS}_{D}(\wedge(n)) = \underbrace{ \mathsf{OS}_{DA}(\wedge_1) \boxtimes \cdot \cdot \cdot \boxtimes \mathsf{OS}_{DA}(\wedge_1) \boxtimes \mathsf{OS}_{D}(\wedge_1)}_{n \text{ times } }\]
where we are suppressing the number of strands in each tangle. Note that the final cap gives a type D module instead of a $DA$ bimodule because there are no incoming strands, and $A(0,0)=\mathbb{Q}$.

Iterating $\psi$, each idempotent state $\mathbf{y}$ in $\mathbf{I}(A_2)$ gives an idempotent state $\y_i$ in $\mathbf{I}(A(n-i,n-i))$ for $1\le i\le n-1$. The only idempotent state for which $\y_i$ is allowable for every $i$ is  $\mathbf{x}_{even}=\{2,4,...,2n\}$, as in Figure \ref{NCapsGenerator}. Thus, $Z\boxtimes ...\boxtimes Z$ is the only generator that survives in the tensor product 
\[ \mathsf{OS}_{D}(\wedge(n)) =  \mathsf{OS}_{DA}(\wedge_1) \boxtimes \cdot \cdot \cdot \boxtimes \mathsf{OS}_{DA}(\wedge_1) \boxtimes \mathsf{OS}_{D}(\wedge_1).\]
We denote this generator by $\mathbf{Z}$. Note that $\delta^1(\mathbf{Z})=0$.

%

 \begin{figure}[ht]
\centering
\def\svgwidth{12cm} \scriptsize
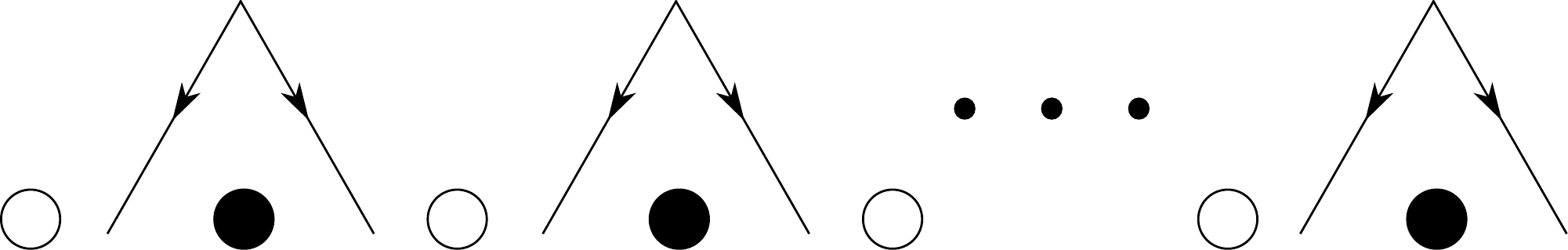
\caption{The allowed idempotent $\mathbf{x}_{even}$.}\label{NCapsGenerator}
\end{figure}


\noindent
Thus, we have the following lemma:

\begin{lemma}
There is an isomorphism $\mathsf{OS}_{DA}(\wedge(n)) \cong A  \cdot \mathbf{I}_{\mathbf{x}_{even}} $, where $A=A(n,n)$.
\end{lemma}

\subsection{The Ozsv\'{a}th-Szab\'{o} $DD$-bimodule for a positive crossing}
Let $\sigma_{i}$ denote the elementary braid on $2n$ strands with a positive crossing between strands $i$ and $i+1$, and let $A=A(n,k)$, $A'=A'(n,2n-k)$. Ozsv\'{a}th and Szab\'{o} define a $DA$ bimodule ${}^{A}\mathsf{OS}_{DA}(\sigma_{i})_{A}$ as well as a $DD$-bimodule ${}^{A,A'}\mathsf{OS}_{DD}(\sigma_{i})$. They are related via box tensor product with the identity $DD$-bimodule:
\[ {}^{A}\mathsf{OS}_{DA}(\sigma_{i})_{A} \hspace{1mm} \boxtimes \hspace{1mm} {}^{A,A'}\mathsf{OS}_{DD}(\id) \cong {}^{A,A'}\mathsf{OS}_{DD}(\sigma_{i})  \] 

We will show that bimodules for crossings are related to the Ozsv\'{a}th-Szab\'{o} bimodules for crossings via homological perturbation, which is much easier to describe for $DD$-bimodules than for $DA$ bimodules. Since ${}^{A,A'}\mathsf{OS}_{DD}(\id)$ is quasi-invertible, showing that our $DD$-bimodule is homotopy equivalent to ${}^{A,A'}\mathsf{OS}_{DD}(\sigma_{i})$ will imply that our $DA$ bimodule is homotopy equivalent to ${}^{A}\mathsf{OS}_{DA}(\sigma_{i})_{A}$. For this reason, we will only give the definition of the $DD$-bimodule ${}^{A,A'}\mathsf{OS}_{DD}(\sigma_{i})$.

Let $\mathbf{x}$ be an idempotent state in $A$ and $\mathbf{y}$ an idempotent state in $A'$. The $DD$-bimodule ${}^{A,A'}\mathsf{OS}_{DD}(\sigma_{i})$ is generated by pairs 
$\mathbf{I_x} \otimes \mathbf{I_y}$ such that either

\begin{enumerate}
    \item $\mathbf{x} \cap \mathbf{y} = \emptyset$, or
    \item $\mathbf{x} \cap \mathbf{y} = \{i+1\}$ \quad and \quad $\mathbf{\overline{x}} \cap \mathbf{\overline{y}} = \{i\}$ or $\{i+2\}$.
\end{enumerate}

\noindent
These pairs can be conveniently grouped into four main types $N, S, E, W$, described below:

\[
\begin{split}
     &\mathbf{N} = \sum_{\substack{ \mathbf{x} \cap \mathbf{y} = \emptyset \\ i+1 \in \mathbf{x}}} \mathbf{I_x} \otimes \mathbf{I_y} \\ 
     &\mathbf{S} = \sum_{\substack{ \mathbf{x} \cap \mathbf{y} = \emptyset \\ i+1 \in \mathbf{y}}} \mathbf{I_x} \otimes \mathbf{I_y} \\
     &\mathbf{E} = \sum_{\substack{ \mathbf{x} \cap \mathbf{y} =\{i+1\} \\ \mathbf{\overline{x}} \cap \mathbf{\overline{y}} = \{i+2\} \\}} \mathbf{I_x} \otimes \mathbf{I_y} \\
     &\mathbf{W} = \sum_{\substack{ \mathbf{x} \cap \mathbf{y} =\{i+1\} \\ \mathbf{\overline{x}} \cap \mathbf{\overline{y}} = \{i\} \\}} \mathbf{I_x} \otimes \mathbf{I_y} \\
\end{split}
\]

Define $\tau_i: [2n] \to [2n]$ to be the transposition that switches $i$ with $i+1$ (and is identity for $j \ne i,i+1$). The differential consists of three types of terms:

\begin{enumerate}
\item 
  $R_j\otimes L_j$
  and $L_j\otimes R_j$ for all $j\in [2n]\setminus \{i,i+1\}$; these connect
  generators of the same type. 
\item
  $-U_{j}\otimes E_{\tau(j)}$ for all $j\in [2n]$; these connect generators of the same type.
\item 
  Terms in the diagram below connect  generators
  of different types:
  \begin{equation*}
    \label{eq:PositiveCrossing}
    \begin{tikzpicture}[scale=2.7]
    \node at (0,3) (N) {$\n$} ;
    \node at (-2,2) (W) {$\w$} ;
    \node at (2,2) (E) {$\e$} ;
    \node at (0,1) (S) {$\s$} ;
    \draw[->] (S) [bend left=7] to node[below,sloped] {\tiny{$-R_i\otimes U_{i+1}-L_{i+1}\otimes R_{i+1}'R_i'$}}  (W)  ;
    \draw[->] (W) [bend left=7] to node[above,sloped] {\tiny{$-L_{i}\otimes 1$}}  (S)  ;
    \draw[->] (E)[bend right=7] to node[above,sloped] {\tiny{$R_{i+1}\otimes 1$}}  (S)  ;
    \draw[->] (S)[bend right=7] to node[below,sloped] {\tiny{$L_{i+1}\otimes U_i + R_i \otimes L_{i}' L_{i+1}'$}} (E) ;
    \draw[->] (W)[bend right=7] to node[below,sloped] {\tiny{$1\otimes L_i'$}} (N) ;
    \draw[->] (N)[bend right=7] to node[above,sloped] {\tiny{$U_{i+1}\otimes R_i' + R_{i+1} R_i \otimes L_{i+1}'$}} (W) ;
    \draw[->] (E)[bend left=7] to node[below,sloped]{\tiny{$1\otimes R_{i+1}'$}} (N) ;
    \draw[->] (N)[bend left=7] to node[above,sloped]{\tiny{$U_{i}\otimes L_{i+1}' + L_{i} L_{i+1}\otimes R_i'$}} (E) ;
  \end{tikzpicture}
\end{equation*}
\end{enumerate}

\subsection{The Ozsv\'{a}th-Szab\'{o} $DD$-bimodule for a negative crossing}

Let $\sigma_{i}^{-1}$ be the elementary braid on $2n$ strands with a negative crossing between strands $i$ and $i+1$, and let $A=A(n,k)$, $A'=A'(n,2n-k)$. Ozsv\'{a}th and Szab\'{o} define both a $DA$ bimodule and a $DD$-bimodule for $\sigma_{i}^{-1}$, which are related via box tensor product with the identity $DD$-bimodule
\[ {}^{A}\mathsf{OS}_{DA}(\sigma^{-1}_{i})_{A} \hspace{1mm} \boxtimes \hspace{1mm} {}^{A,A'}\mathsf{OS}_{DD}(\id) \cong {}^{A,A'}\mathsf{OS}_{DD}(\sigma^{-1}_{i}).  \] 

As with the positive crossing, we will only describe the $DD$-bimodule ${}^{A,A'}\mathsf{OS}_{DD}(\sigma^{-1}_{i})$, since this is the version we will use to compare with our bimodule.

The $DD$-bimodule ${}^{A,A'}\mathsf{OS}_{DD}(\sigma^{-1}_{i})$ has the same generators as the positive crossing bimodule ${}^{A,A'}\mathsf{OS}_{DD}(\sigma_{i})$. The differential consists of three types of terms:

\begin{enumerate}
\item 
  $R_j\otimes L_j$
  and $L_j\otimes R_j$ for all $j\in [2n]\setminus \{i,i+1\}$; these connect
  generators of the same type. 
\item
  $-U_{j}\otimes E_{\tau(j)}$ for all $j\in [2n]$; these connect generators of the same type.
\item   Terms in the diagram below connect  generators
  of different types:
  \begin{equation*}
\begin{tikzpicture}[scale=2.7]
    \node at (0,3) (N) {$\n$} ;
    \node at (-2,2) (W) {$\w$} ;
    \node at (2,2) (E) {$\e$} ;
    \node at (0,1) (S) {$\s$} ;
    \draw[->] (W) [bend left=7] to node[above,sloped] {\tiny{${U_{i+1}}\otimes{L_i'}+{L_{i} L_{i+1}}\otimes{R_{i+1}'}$}}  (N)  ;
    \draw[->] (N) [bend left=7] to node[below,sloped] {\tiny{${1}\otimes{R_{i}'}$}}  (W)  ;
    \draw[->] (N)[bend right=7] to node[below,sloped] {\tiny{${1}\otimes{L_{i+1}'}$}}  (E)  ;
    \draw[->] (E)[bend right=7] to node[above,sloped] {\tiny{${U_i}\otimes{R_{i+1}'} + {R_{i+1} R_{i}}\otimes{L_i'}$}} (N) ;
    \draw[->] (S)[bend right=7] to node[above,sloped] {\tiny{$-{R_i}\otimes{1}$}} (W) ;
    \draw[->] (W)[bend right=7] to node[below,sloped] {\tiny{$-{L_i}\otimes{U_{i+1}} - {R_{i+1}}\otimes{L_{i}' L_{i+1}'}$}} (S) ;
    \draw[->] (S)[bend left=7] to node[above,sloped]{\tiny{${L_{i+1}}\otimes{1}$}} (E) ;
    \draw[->] (E)[bend left=7] to node[below,sloped]{\tiny{${R_{i+1}}\otimes{U_{i}} + {L_i}\otimes{R_{i+1}' R_{i}'}$}} (S) ;
  \end{tikzpicture}
  \end{equation*}
\end{enumerate}

\subsection{The Ozsv\'{a}th-Szab\'{o} $DA$ bimodule for a minimum}

Let $\vee_{c}$ denote the $(2n,2n+2)$ tangle which has a minimum between strands $c$ and $c+1$. The corresponding algebras are $A_{1}=A(n+1,k+1)$ and $A_{2}=A(n,k)$. Ozsv\'{a}th and Szab\'{o} only define the $DA$ bimodule for $\vee_{1}$, as any minimum can be moved to the left via Reidemeister II moves. Fortunately, in a plat diagram, the minima can be ordered so that each is leftmost in its elementary diagram (see Figures \ref{NCups2} and \ref{orderedcups}), so this case is sufficient for comparing to our construction.

 \begin{figure}[ht]
\centering
\def\svgwidth{10cm} \scriptsize
\input{figs/NCups.pdf_tex}
\caption{The diagram $\vee(n)$ for the plat minimum.}\label{NCups2}
\end{figure}
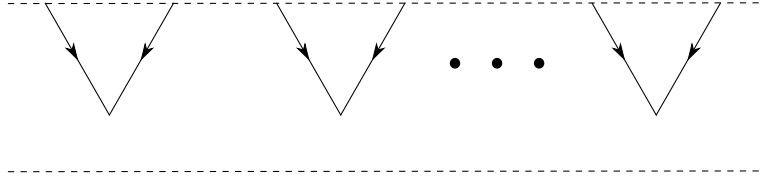

 \begin{figure}[ht]
\centering
\def\svgwidth{10cm} \scriptsize
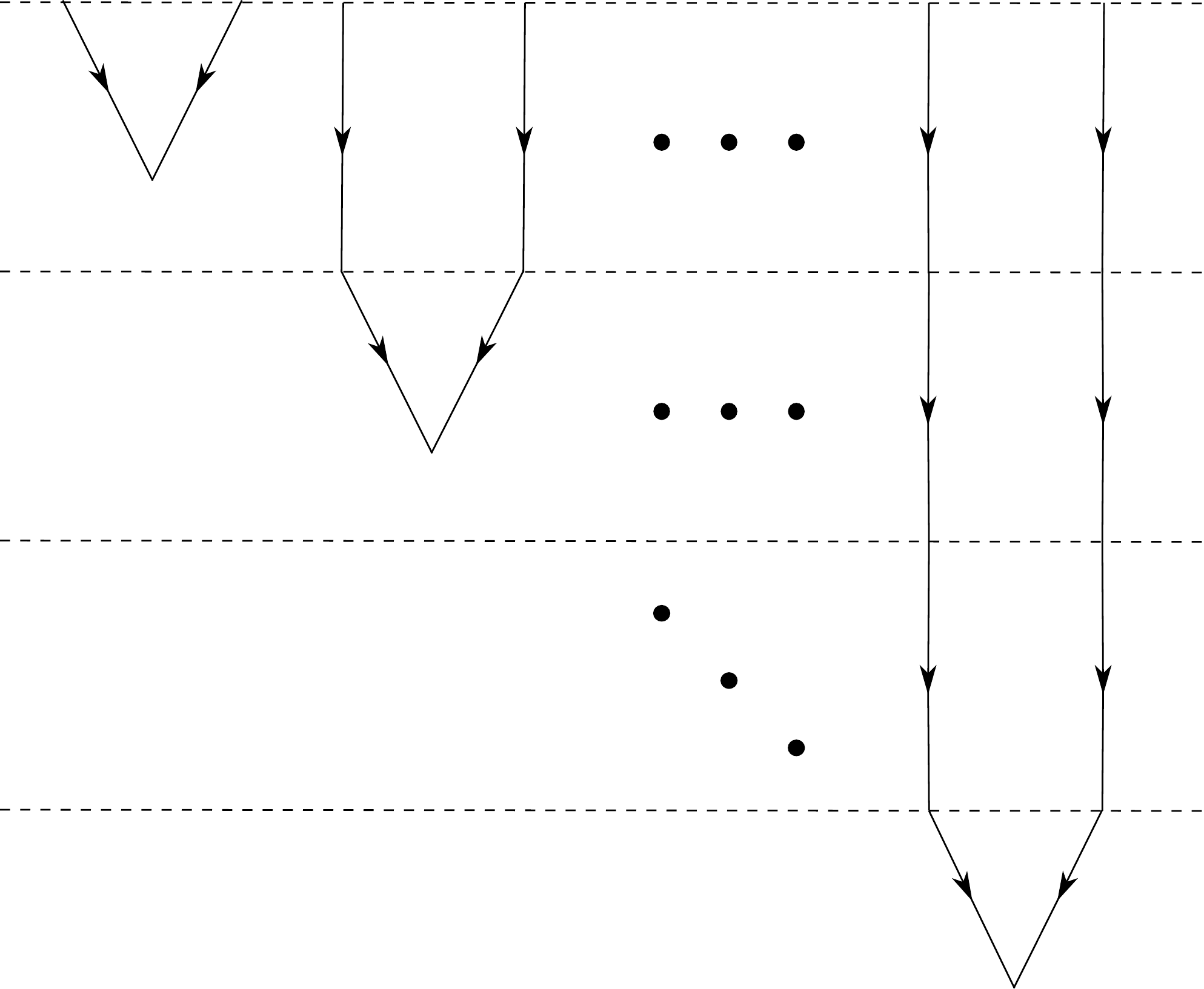
\caption{The plat minimum diagram with each minimum occurring leftmost in its slice.}\label{orderedcups}
\end{figure}

Ozsv\'{a}th and Szab\'{o} define two different bimodules for $\vee_{1}$, depending on whether or not $n=0$, i.e. they treat the absolute minimum differently from the other minima. This comes from the fact that a Kauffman states diagram requires a choice of marked edge in $D$, which they choose to be at the absolute minimum. The version that we give here will NOT treat the $n=0$ case as special -- we will use the same bimodule for the absolute minimum as for the relative minima. We expect that this modification should replace the knot Floer complex for $D$ with the knot Floer complex for $D \sqcup \text{unknot}$. This  typically what happens when one forgets about a decorated edge, as the decoration has been moved to the unknotted component.

We give the `alternative construction' of ${}^{A_2}\mathsf{OS}(\vee_{1})_{A_1}$ from \cite{ozsvath2018kauffman}, Section 9.2. Let 
\[\mathbf{I}=\sum_{\substack{1\notin\mathbf{x} \\ 2 \in \mathbf{x}}} \mathbf{I_x} . \] 

\noindent
There is an inclusion $\phi: A_2 \to \mathbf{I} \cdot A_1 \cdot \mathbf{I}$ which maps into the portion of $A_1$ with $w_1=w_2=0$. Concretely, $\phi(L_i)=\mathbf{I} \cdot L_{i+2}$ and $\phi(R_i)=\mathbf{I} \cdot R_{i+2}$. Thus, $\mathbf{I} \cdot A_1$ can be viewed as an $A_2 \text{--} A_1$-bimodule, where the left action is given by $m_{1|1|0}(a,b) = \phi(a)  b$ and the right action is just right multiplication $m_{0|1|1}(a,b) =a b$. 

With this convention, the $AA$-bimodule for $\vee_1$ is given by 
\[{}_{A_2}\mathsf{OS}_{AA}(\vee_{1})_{A_1} = \mathbf{I} \cdot A_1 /(L_1 L_2 \cdot A_{1}) \]

\subsection{The Ozsv\'{a}th-Szab\'{o} type A module for the plat minimum $\vee(n)$}
Consider the plat minimum diagram consisting of $n$ cups in Figure \ref{NCups2}. We will denote this diagram $\vee(n)$. The type A module $\mathsf{OS}(\vee(n))$ is defined over $A=A(n,n)$.

Suppose we order the minima from left to right so that the left minimum is highest in the diagram and the right minimum is lowest. Then 
\[ \vee(n) = \underbrace{\vee_1 \ast \vee_1 \ast ... \ast \vee_1}_{n \text{ times } }\] and 
\[ \mathsf{OS}_{A}(\vee(n)) = \underbrace{ \mathsf{OS}_{A}(\vee_1) \boxtimes \mathsf{OS}_{DA}(\vee_1) \boxtimes ... \boxtimes \mathsf{OS}_{DA}(\vee_1)}_{n \text{ times } }\]
where for simplicity, we are suppressing the number of strands in each tangle from the notation.

As in the plat maximum diagram, this box tensor product picks out the idempotent $\mathbf{I}_{\mathbf{x}_{even}}$, where $\mathbf{x}_{even}=\{2,4,...,2n\}$. 

\begin{lemma}

There is an isomorphism $\mathsf{OS}_{A}(\vee(n)) \cong \mathbf{I}_{\mathbf{x}_{even}} \cdot A$, where $A=A(n,n)$.

\end{lemma}

\begin{proof}

It follows from the definition of $\mathsf{OS}_{DA}(\vee_1)$ that
\[\mathsf{OS}_{A}(\vee(n)) \cong \mathbf{I}_{\mathbf{x}_{even}} \cdot A /( L_{1}L_{2} \cdot A,L_{3}L_{4} \cdot A,..., L_{2n-1}L_{2n} \cdot A ) \]

\noindent
But for $i=1,...,n$, we already have $\mathbf{I}_{\mathbf{x}_{even}} \cdot L_{2i-1} L_{2i} =0$ for idempotent reasons. Thus, the relations are redundant, and 
\[\mathsf{OS}_{A}(\vee(n)) \cong \mathbf{I}_{\mathbf{x}_{even}} \cdot A \]
as desired.
\end{proof}

\section{Background II: Our bimodules and Khovanov homology}

In this section, we will describe the bimodules from \cite{alishahi2018link} and their relationship with Khovanov homology. The objects in this section are true bimodules (as opposed to $DA$- or $DD$-bimodules) over an algebra $\mathcal{A}$ which is isomorphic to $A$.

\subsection{The algebra $\mathcal{A}$}

The algebra $\mathcal{A}$ is isomorphic to $A=A(n,k)$, but will be described in terms of dual generators. In particular, for each idempotent state $\mathbf{x}$ we have an idempotent
\[  \iota_{\mathbf{x}} = \mathbf{I_{\overline{x}}}  \]
where $\mathbf{\overline{x}}$ is the complement of $\mathbf{x}$. With respect to these complementary idempotents, there are generators $\mathcal{L}_{i}, \mathcal{R}_{i}$ given by 
\[ \mathcal{L}_{i} = R_i \hspace{2cm} \mathcal{R}_i = L_i \]

The algebra $\mathcal{A}$ is generated by the idempotents $\iota_{\mathbf{x}}$ together with the $\mathcal{L}_i, \mathcal{R}_i, U_{i}$, modulo the relations
\[ \mathcal{R}_{i+1} \cdot \mathcal{R}_{i} = 0 \] \[\mathcal{L}_{i} \cdot \mathcal{L}_{i+1} = 0 \]
\[ \mathcal{L}_{i} \cdot \mathcal{R}_{i} = \iota_{\mathbf{x(i+1)}} \cdot U_{i} \cdot \iota_{\mathbf{x(i+1)}}  \]
\[ \mathcal{R}_{i} \cdot \mathcal{L}_{i} = \iota_{\mathbf{x(i)}} \cdot U_{i} \cdot \iota_{\mathbf{x(i)}} \]

\subsection{Interpretation as strands algebras}

The Ozsv\'{a}th-Szab\'{o} idempotent states correspond to local Kauffman states, so $i \in \mathbf{x}$ means that the $i$th slot between the strands is occupied. Our algebras came from a planar Heegaard diagram where it is more natural to view the \emph{strands} as being occupied, so in $\A$, $i \in \mathbf{x}$ will be interpreted as the $i$th strand is occupied (see Figures \ref{R1OS} and \ref{R1AD}).

\begin{figure}[ht]
\centering
\def\svgwidth{8cm} \scriptsize
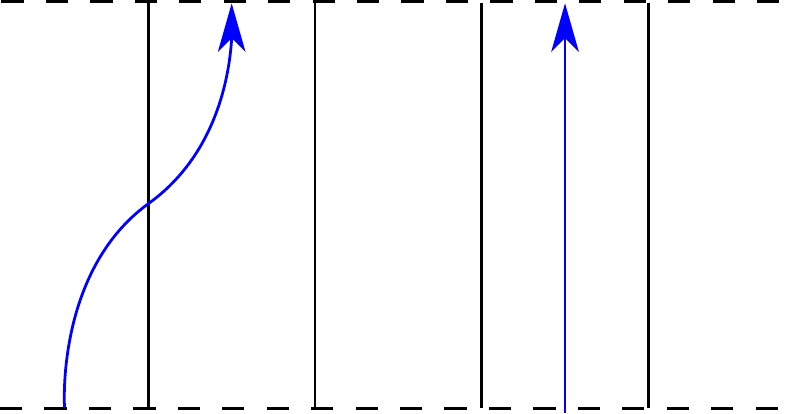
\caption{Strands representation of the element $R_{1}^{\mathbf{x}} \in A(2,2)$ for $\mathbf{x} = \{1,4\}$.}\label{R1OS}
\end{figure}

\begin{figure}[ht]
\centering
\def\svgwidth{8cm} \scriptsize
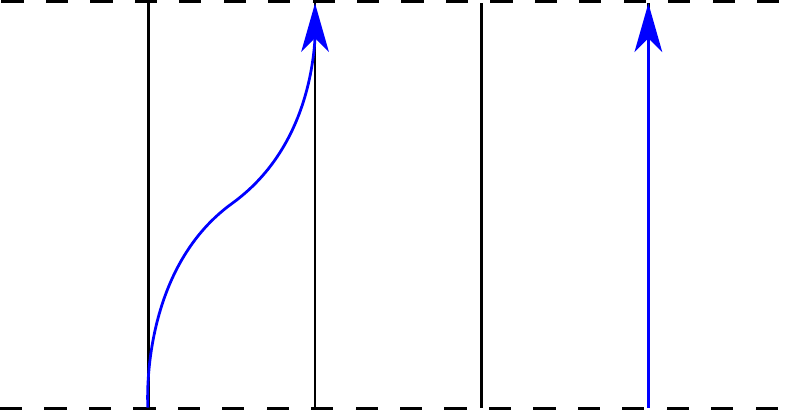
\caption{Strands representation of the element $\mathcal{R}_{1}^{\mathbf{x}} \in \A(2,2)$ for $\mathbf{x} = \{1,4\}$.}\label{R1AD}
\end{figure}

\subsection{The identity bimodule}

Let $\id$ denote the identity tangle on $2n$ strands. Then the bimodule $\M(\id)$ is the free $\A$-module of rank $1$, where the left and right actions are the standard left and right multiplication in $\A$.

This bimodule has the following geometric interpretation, which will be useful for understanding more complicated tangles. Viewing $\id$ as an oriented graph with boundary, let $Z$ denote an oriented $k$-component cycle in $\id$.

\begin{definition}

Given a $k$-component cycle $Z$ in an open braid $T$, let $b(Z)$ (resp. $t(Z)$) denote the idempotent state consisting of incoming strands (resp. outgoing strands) which belong to $Z$.

\end{definition}

Each cycle $Z$ gives a generator $x_{Z}$ of $\M(\id)$ with left idempotent $\iota_{b(Z)}$ and right idempotent $\iota_{t(Z)}$, $x_{Z} = \iota_{b(Z)} \cdot x_{Z} \cdot \iota_{t(Z)}$. Note that in the case of the identity tangle, $b(Z)=t(Z)$. However, this will not be true in general.

The bimodule $\M(\id)$ generated by the $x_{Z}$ over all cycles $Z$ modulo the relations necessary to make the following map an isomorphism:
\[ f: \M(\id) \to \A \quad  \text{defined by} \quad f(x_{Z}) = \iota_{b(Z)}\]

\noindent
The interpretation of these relations in terms of the strands algebras are given in Figure \ref{AlgebraMultiplicationFig}. They are: 

\begin{itemize}
    \item Pushing strands through the diagram.
    \item No self-intersection.
    \item Multiplication on a strand.
\end{itemize}

\begin{figure}[ht]
\centering
\def\svgwidth{12cm} 
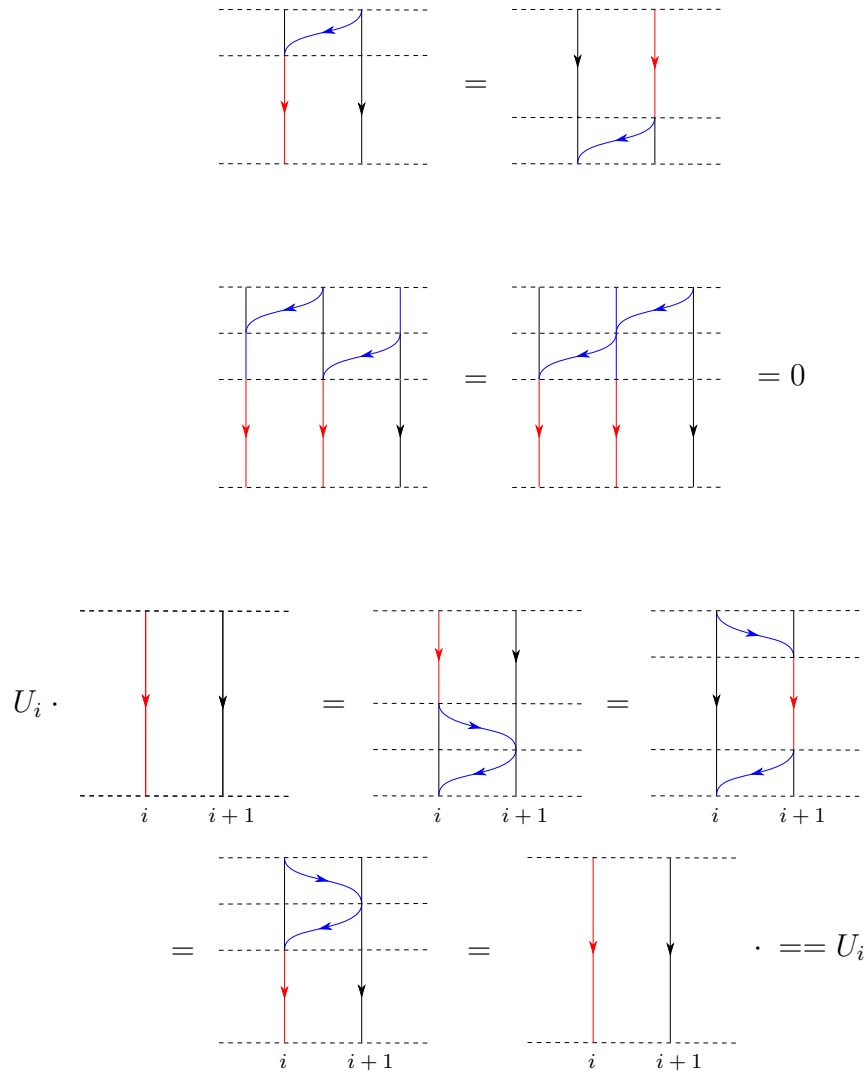
\caption{A diagrammatic description of some of the relations in the strands algebra.}\label{AlgebraMultiplicationFig}
\end{figure}

\subsection{Our bimodule for singularization}
\label{ex:sing}
 Let $\mathsf{X}_i$ denote the elementary singular braid on $2n$ strands with a singularization between strands $i$ and $i+1$. Let $e_1$ and $e_2$ (resp. $e_3$ and $e_4$) denote the left and right incoming (resp. outgoing) edges at the only $4$-valent vertex of $\mathsf{X}_i$, respectively (see Figure \ref{xi}). 
 
Let $Z$ be a $k$-component cycle in $\mathsf{X}_i$ which does not include all four edges $e_1,e_2,e_{3},e_4$. The bimodule $\M(\mathsf{X}_i)$ is generated by $x_{Z}$ over all such $Z$, modulo the relations below. For each subset $I\subset\{1,2,3,4\}$, let $\mathcal{Z}I$ denote the set of cycles that locally consist of the edges labelled with elements in $I$.

\begin{figure}[ht]
\centering
\def\svgwidth{12cm} \scriptsize
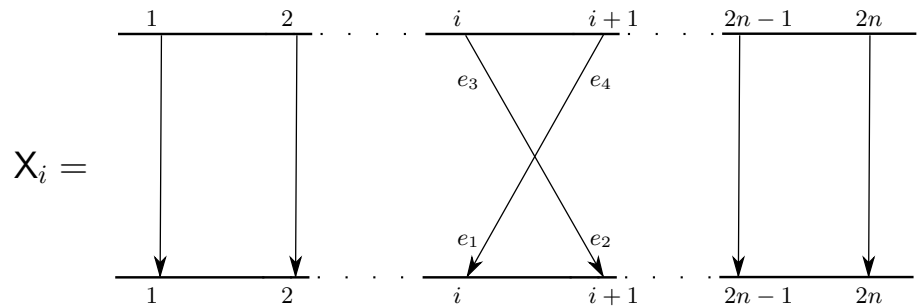
\caption{The elementary singular braid $\mathsf{X}_{i}$.}\label{xi}
\end{figure}

 \begin{enumerate}
 \item $x_{Z}=\iota_{b(Z)} \cdot x_{Z} \cdot \iota_{t(Z)}$,
 \item For $j<i$ or $j>i+1$, $U_j \cdot x_Z=x_Z \cdot U_j$, 
 \item For $j<i-1$ or $j>i+1$, if $j+1\in b(Z)$, then $\mathcal{L}_{j}x_{l_j(Z)}=x_{Z} \mathcal{L}_{j}$,
 \item For $j<i-1$ or $j>i+1$, if $j\in b(Z)$, then $\mathcal{R}_{j}x_{r_{j}(Z)}=x_{Z} \mathcal{R}_{j}$,
 \item $(U_i+U_{i+1})\cdot x_Z=x_Z \cdot (U_i+U_{i+1})$,
 \item $(U_i U_{i+1}) \cdot x_{Z}=x_Z \cdot (U_i U_{i+1})$,
 \item For $Z\in\mathcal{Z}\emptyset$, if $i-1\in b(Z)$, then 
 \[ \mathcal{L}_{i-1} \cdot x_{Z}=x_{Z(1,3)} \cdot \mathcal{L}_{i-1} \ \ \ \ \text{and}\ \ \ \ x_Z\cdot \mathcal{R}_{i-1}= \mathcal{R}_{i-1} \cdot x_{Z(1,3)},\]
  where $Z(1,3)\in\mathcal{Z}\{1,3\}$ denotes the cycle obtained from $Z$ by replacing the $(i-1)$-th strand with local cycle $e_1e_3$.
 \item For $Z\in\mathcal{Z}\{1,3\}$, \[\mathcal{L}_i \cdot x_Z=x_{Z(2,3)}\ \ \ \text{and}\ \ \ x_{Z} \cdot \mathcal{R}_i=x_{Z(1,4)},\]
 where $Z(2,3)=(Z\setminus e_1)\cup e_2$ and $Z(1,4)=(Z\setminus e_3)\cup e_4$.
 \item For $Z\in\mathcal{Z}\{1,4\}$, $\mathcal{L}_i \cdot x_Z=x_{Z(2,4)}$ where $Z(2,4)=(Z\setminus e_1)\cup e_2$.
 \item For $Z\in\mathcal{Z}\{2,3\}$, $x_Z \cdot \mathcal{R}_i=x_{Z(2,4)}$ where $Z(2,4)=(Z\setminus e_3)\cup e_4$.
 \item For each $Z\in\mathcal{Z}\emptyset$, if $i+2\in b(Z)$ then, 
 \[ \begin{split}
 &x_Z \cdot \mathcal{L}_{i+1} \mathcal{L}_i=\mathcal{L}_{i+1} \cdot x_{Z}(2,3), \ \ \ \ x_Z \cdot U_{i} \mathcal{L}_{i+1}=\mathcal{L}_{i+1} \cdot x_{Z(2,4)}\\
 & \mathcal{R}_i \mathcal{R}_{i+1} \cdot x_Z =  x_{Z(1,4)} \cdot \mathcal{R}_{i+1}\ \ \ \ \mathcal{R}_{i+1} U_i \cdot x_Z=x_{Z(2,4)}R_{i+1}.
 \end{split}
 \]
where $Z(i,j)$ denotes the cycle obtained from $Z$ by replacing the $(i+2)$-th strand with local cycle $e_ie_j$. 
\end{enumerate}

We will now give a description of $\M(\mathsf{X}_i)$ as a free left $\A$-module. 

\begin{definition}
Let $x_{\emptyset}=\sum_{Z\in \mathcal{Z}\emptyset}x_Z$ and $x_{13}=\sum_{Z\in\mathcal{Z}\{1,3\}}x_Z$. We define, 
\[N_{+}=x_{\emptyset}+x_{13},\quad N_{-}=N_{+}U_{i}-U_{i+1}N_{+}, \quad E=N_{+}L_{i+1}, \quad W=N_{+}R_i.\]  

\end{definition}

\begin{lemma}\label{lem:gen}
Consider the generators 
\[ \mathcal{G} = \{ N_{+}, N_{-}, W, E\}.\] Then $\A \mathcal{G} = \mathsf{M}(\mathsf{X}_i)$, i.e. these four generators generate $\mathsf{M}(\mathsf{X}_i)$ as a left module over $\A$. 

\end{lemma}

\begin{proof}

Consider the generators $x_{Z}$. Since $\A \{x_{Z}\} \A =\mathsf{M}(\mathsf{X}_i)$, it suffices to show that:

\begin{enumerate}[label=(\alph*)]
\item Each $x_{Z}$ is in $\A \mathcal{G}$, and
\item That it is closed under right multiplication, i.e. that for any $x \in \mathcal{G}$ and any generator $a$ of $\A$, $xa=by$ for some $y \in \mathcal{G}$ and some $b \in \A$.
\end{enumerate}

For $(a)$, the generators $x_{Z}$ where $Z\in\mathcal{Z}\emptyset\cup\mathcal{Z}\{1,3\}\cup\mathcal{Z}\{1,4\}$ are already included in $\mathcal{G}$. Since $\mathcal{L}_{i}N_{+} = \sum_{Z\in\mathcal{Z}\{2,3\}}x_Z$ and $\mathcal{L}_{i}W= \sum_{Z\in\mathcal{Z}\{2,4\}}x_{Z}$, this covers all cycles $Z$.

For $(b)$, let $x \in \mathcal{G}$ and let $a$ be a generator of $\A$, so $a = \mathcal{L}_{j}$, $\mathcal{R}_{j}$, or $U_{j}$ for some $j$. 


If $j<i$ or $j>i+1$, $xa=ax$ for each $x\in\mathcal{G}$. For $j=i$ and $j=i+1$, the right multiplication is given by Table \ref{SingTable}, where the left column indicates the input generator and the top row is the algebra element which is multiplying on the right. For example, the element in the second column, fourth row is saying that $EU_{i} = \mathcal{L}_{i+1}\mathcal{L}_{i} W$.

\begin{table}[!h]
\resizebox{\textwidth}{!}{%
\begin{tabular}{|l|l|l|l|l|l|l|} 
\hline 
 & $U_{i}$                        & $U_{i+1}$                        & $\L_{i}$            & $\L_{i+1}$              & $\R_{i}$               & $\R_{i+1}$              \\ \hline
$N_{+}$   & $N_{-}+U_{i+1}N_{+}$             & $-N_{-}+U_{i}N_{+}$            & $0$                & $E$                  & $W$                   & $0$                  \\ \hline
$N_{-}$   & $U_{i}N_{-}$                   & $U_{i+1}N_{-}$                   & $0$                & $\L_{i+1}\L_{i}W-U_{i+1}E$ & $-\R_{i}\R_{i+1}E+U_{i}W$ & $0$                  \\ \hline
$E$       & $\L_{i+1}\L_{i}W$                  & $(U_{i}+U_{i+1})E - \L_{i+1}\L_{i}W$ & $\L_{i+1}\L_{i}N_{+}$  & $0$                  & $0$                   & $U_{i}N_{+} - N_{-}$ \\ \hline
$W$       & $(U_{i}+U_{i+1})W - \R_{i}\R_{i+1}E$ & $\R_{i}\R_{i+1}E$                  & $N_{-}+U_{i+1}N_{+}$ & $0$                  & $0$                   & $\R_{i}\R_{i+1}N_{+}$    \\ \hline
\end{tabular}}
\vspace{2mm}
\caption{Right multiplication table for $\M(\mathsf{X}_i)$}\label{SingTable}
\end{table}

\end{proof}

\begin{lemma}\label{lem:rightu}
The left submodules $\Alg \cdot N_{+}$ and $\Alg \cdot N_{+}U_i$ are disjoint. 
\end{lemma}

\begin{proof}
By the MOY \rom{2} relation from \cite{alishahi2018link}, we know that 
\[\mathsf{M}(\mathsf{X}_i)\otimes_{\A}\mathsf{M}(\mathsf{X}_i)\cong\mathsf{M}(\mathsf{X}_i)\{1\}\oplus\mathsf{M}(\mathsf{X}_i)\{-1\}. \]
Under this isomorphism, $N_{+} \otimes N_{+}$ and $N_{+} U_i\otimes N_{+}$ are mapped to the element $N_{+}$ in the summands $\mathsf{M}(\mathsf{X}_i)\{1\}$ and $\mathsf{M}(\mathsf{X}_i)\{-1\}$, respectively. Thus, the left submodules $\A \cdot  N_{+} \otimes N_{+}$ and $\A \cdot N_{+} U_i \otimes N_{+}$ and consequently $\Alg \cdot N_{+}$ and $\Alg \cdot N_{+} U_i$ are disjoint.
\end{proof}

\begin{corollary}
The left submodule $\A \cdot N_{+} U_{i+1}$ is also disjoint from $\A \cdot N_{+}$.
\end{corollary}

For simplicity, let 
\[\iota_{i\setminus i+1}=\sum_{ \mathbf{x} \cap\{i,i+1\}=\{i\}}\iota_{\mathbf{x}},\quad \iota_{\emptyset}=\sum_{\mathbf{x} \cap\{i,i+1\}=\emptyset}\iota_{\mathbf{x}},\quad \iota_{i}=\sum_{i\in  \mathbf{x}}\iota_{\mathbf{x}}\]

\begin{lemma}\label{lem:elem}
The elements of $\mathcal{G}$ generate left modules isomorphic to the following idempotent subalgebras:
\begin{itemize}
\item $\Alg \cdot N_{+} \cong \Alg \cdot N_{-} \cong \Alg \cdot \iota_{i\setminus i+1}\oplus \Alg \cdot \iota_{\emptyset}$,
\item $\Alg \cdot  W\cong \Alg \cdot \iota_{i\setminus i+1}$
\item $\Alg \cdot E\cong \Alg \cdot \iota_{i+2}$.
\end{itemize}
\end{lemma}

\begin{proof}
First, we prove that $\A \cdot N_{+}\cong \A \cdot \iota_{i\setminus i+1}\oplus\A \cdot \iota_{\emptyset}$.
Consider the $\A$-module homomorphism 
\[\begin{split}
&f:\A \cdot \iota_{i\setminus i+1}\oplus\A \cdot \iota_{\emptyset}\to \A \cdot N_{+}\\
&f(a\iota_{i\setminus i+1}+b\iota_{\emptyset})=(a+b)N_{+}.
\end{split}
\] 
It is clearly surjective. Let $d^-$ be the edge bimodule homomorphism from $\mathsf{M}(\mathsf{X}_i)$ to $\A$. By definition, $d^-\circ f(\iota_{\bullet})=\iota_{\bullet}$ for $\bullet=\emptyset, i\setminus i+1$. Thus, $d^-\circ f$ is equal to the inclusion of $\A \iota_{i\setminus i+1}\oplus\A \iota_{\emptyset}$ into $\A$, and so $f$ is injective. From Lemma \ref{lem:rightu}, we also have $\A \cdot N_{+}U_i\cong \A \cdot \iota_{i\setminus i+1}\oplus\A \cdot \iota_\emptyset$. Since $N_{-} = N_{+}U_{i} - U_{i+1}N_{+}$, it follows that $\A \cdot N_{-} \cong \A \cdot \iota_{i\setminus i+1}\oplus\A \cdot \iota_\emptyset$ as well.

Similarly, we need to proof that if $aW=0$ for some $a\in \A \cdot \iota_{i\setminus i+1}$, then $a=0$. Since $W=N_+\mathcal{R}_i$ we have 
\[aW\mathcal{L}_i=a N_{+} \iota_{i \setminus i+1} U_i= a \iota_{i \setminus i+1} N_{+}  U_i = 0.\]
Therefore, $a=0$. The proof for $E$ is similar.

\end{proof}

\begin{lemma} \label{freesingular}
As a left $\Alg$-module $\mathsf{M}(\mathsf{X}_i)$ is a projective module isomorphic to \[\lsub{\Alg}\mathsf{M}(\mathsf{X}_i)\cong\Alg \cdot N_{+} \oplus\Alg \cdot N_{-} \oplus\Alg \cdot W\oplus\Alg \cdot E \]
\end{lemma}

\begin{proof}
By Lemma \ref{lem:gen}, $\mathcal{G}$ generates $\mathsf{M}(\mathsf{X}_i)$ as a left $\Alg$-module. Suppose $a_{1}N_++a_{2}N_-+a_{3}E+a_{4}W=0$. Since $N_{+}\iota_{i+1}=N_{-}\iota_{i+1}=0$, while $W\iota_{i+1}=W$ and $E\iota_{i+1}=E$, we have $a_1 N_{+}+a_2N_{-}=0$. Thus, $(a_1-a_2U_{i+1})N_{+}+a_2N_{+}U_i=0$. Lemmas \ref{lem:elem} and \ref{lem:rightu} imply that  
$a_1 N_{+}=a_2 N_{-}=0$.

On the other hand, $(a_3E+a_4W)\mathcal{R}_{i+1}=a_3N_+U_{i+1}\iota_{i+2}+a_4\mathcal{R}_i\mathcal{R}_{i+1}N_+=0$. Again, lemmas \ref{lem:elem} and \ref{lem:rightu} imply that  $a_3\iota_{i+2}=a_3E=0$ and so $a_4W=0$.

\end{proof}

\subsection{Our bimodule for a plat maximum}

Let $\wedge(n)$ denote the plat maximum consisting of $n$ caps. Each cap consists of two edges oriented into a sink -- suppose the edges are numbered $e_{1},...,e_{2n}$ so that $e_{i}$ lies on strand $i$. The bimodule $\M(\wedge(n))$ is an $\A(n,n), \A(0,0)$-bimodule, where $\A(0,0)=\mathbb{Q}$.

The generators of $\M(\wedge(n))$ correspond to $n$-component cycles $Z$ relative to the boundary of $\wedge(n)$ union the bivalent vertices. For the $i$th maximum of $\wedge(n)$, $Z$ contains either $e_{2i-1}$ or $e_{2i}$, but not both. There are $2^n$ such cycles.

The module $\M(\wedge(n))$ is generated as a left module over $\A(n,n)$ by the $x_{Z}$ modulo the following relations:

\begin{enumerate}
    \item If $e_{2i-1} \in Z$ and $Z'=Z \setminus \{e_{2i-1} \} \cup \{e_{2i} \} $, then $\mathcal{L}_{2i-1} \cdot x_{Z} = x_{Z'}$.
    \item For all $i$ and $Z$, $\mathcal{L}_{2i} \cdot x_{Z}=0$.
\end{enumerate}

\noindent
Note that from the first relation, it follows that $\M(\wedge(n))$ is generated by the generator $x_{Z_{odd}}$, where $Z_{odd}$ is the cycle which has the left edge at each maximum:
\[Z_{odd}=\{e_1, e_3, ... e_{2n-1} \} \]

Let $\mathbf{x}_{odd}$ be the idempotent state $\{1,3,..., 2n-1\}$.

\begin{lemma}
The module $\M(\wedge(n))$ is isomorphic to $\A(n,n) \cdot \iota_{\mathbf{x}_{odd}}$.
\end{lemma}

\begin{proof}

Since the module $\M(\wedge(n))$ is generated by the element $x_{Z_{odd}}$ which has idempotent $\iota_{\mathbf{x}_{odd}}$, it is clearly a quotient of $\A(n,n) \cdot \iota_{\mathbf{x}_{odd}}$. To see that there are no additional relations in $\M(\wedge(n))$, we need to show that relation (2), $\mathcal{L}_{2i} \cdot x_{Z}=0$, is true in 
the identification with $\A(n,n) \cdot \iota_{\mathbf{x}_{odd}}$. In other words, for each idempotent state $\mathbf{x}'$ where $\mathbf{x}'\cap\{2i,2i+1\}=\{2i\}$ we have $\mathcal{L}_{2i} \cdot \phi^{\mathbf{x}',\mathbf{x}_{odd}}(1)=0$. But the left idempotent for $\mathcal{L}_{2j}\iota_{\mathbf{x}'}$ is too far from $\mathbf{x}_{odd}$, so this completes the proof.

%
%
%

\end{proof}

\subsection{Our bimodule for a positive crossing}

Let $\sigma_i$ be an elementary positive braid with $2n$ strands where the crossing is between the $i$ and $i+1$ strands. The complex of bimodules $\mathsf{M}(\sigma_i)$ is defined to be the mapping cone of the $\Alg$-bimodule homomorphism
\[d^+:\M(\id) \to \mathsf{M}(\mathsf{X}_i)\]
defined as follows. 
\[d^+(\iota_{\mathbf{x}})=\begin{cases}
\begin{array}{ll}
x_{Z(\mathbf{x})}U_i-U_{i+1}x_{Z(\mathbf{x})}&\text{if}\ \mathbf{x}\cap\{i,i+1\}=\{i\}\\
x_{Z(\mathbf{x})}- \mathcal{R}_{i+1}x_{Z(\mathbf{y})} \mathcal{L}_{i+1}&\text{if}\ \mathbf{x}\cap\{i,i+1\}=\{i+1\}\\
- \mathcal{R}_{i+1}x_{Z(\mathbf{y})}\mathcal{L}_{i+1}&\text{if}\ \mathbf{x}\cap\{i,i+1\}=\{i,i+1\}
\end{array}
\end{cases}
\]
where, $Z(\mathbf{x})$ is the cycle where $b(Z(\mathbf{x}))=t(Z(\mathbf{x}))=\mathbf{x}$ and $\mathbf{y}=r_{i+1}(\mathbf{x})$. Note that if $\mathbf{y}$ is not defined, $x_{Z(\mathbf{y})}=0$. 

Note that we can split the algebra $\A$ along idempotents as \[ \mathcal{A} = \A \cdot \iota_{ i \setminus i+1} \oplus \A \cdot \iota_{\emptyset} \oplus \A \cdot \iota_{i+1} \]

\noindent
Let $N_{0} = \iota_{i \setminus i+1} + \iota_{\emptyset}$ and $S = \iota _{i+1}$ so that $\M(\id) = \A \cdot N_{0} \oplus \A \cdot S$. The generator $N_{0}$ has the same idempotent as $N_{+}$ and $N_{-}$, so that $\A \cdot N_{0} \cong \A \cdot N_{+} \cong \A \cdot N_{-}$.

We can give a description of the bimodule $\M(\sigma_i)$ as a left $\A$-module over the six generators $\{ N_{+},N_{-}, E, W, N_{0}, S \}$. The differential is given by 
\[ d^{+}(N_{0})=N_{-} \]
\[ d^{+}(S) = \mathcal{L}_{i} W- \mathcal{R}_{i+1} E   \]

\noindent
The relevant right multiplication maps are described in the following table. Note that for $j \ne i,i+1$, $\mathcal{R}_j$,  $\mathcal{L}_j$, and $U_j$ each commute with all six generators of $\M(\sigma_i)$ -- $\mathcal{R}_j$ and $ \mathcal{L}_j$ just change the idempotent away from $i,i+1$.

\begin{table}[!h]
\resizebox{\textwidth}{!}{%
\begin{tabular}{|l|l|l|l|l|l|l|} 
\hline 
 & $U_{i}$                        & $U_{i+1}$                        & $\L_{i}$            & $\L_{i+1}$              & $\R_{i}$               & $\R_{i+1}$              \\ \hline
$N_{+}$   & $N_{-}+U_{i+1}N_{+}$             & $-N_{-}+U_{i}N_{+}$            & $0$                & $E$                  & $W$                   & $0$                  \\ \hline
$N_{-}$   & $U_{i}N_{-}$                   & $U_{i+1}N_{-}$                   & $0$                & $\L_{i+1}\L_{i}W-U_{i+1}E$ & $-\R_{i}\R_{i+1}E+U_{i}W$ & $0$                  \\ \hline
$E$       & $\L_{i+1}\L_{i}W$                  & $(U_{i}+U_{i+1})E - \L_{i+1}\L_{i}W$ & $\L_{i+1}\L_{i}N_{+}$  & $0$                  & $0$                   & $U_{i}N_{+} - N_{-}$ \\ \hline
$W$       & $(U_{i}+U_{i+1})W - \R_{i}\R_{i+1}E$ & $\R_{i}\R_{i+1}E$                  & $N_{-}+U_{i+1}N_{+}$ & $0$                  & $0$                   & $\R_{i}\R_{i+1}N_{+}$    \\ \hline
$N_{0}$   & $U_{i}N_{0}$                   & $U_{i+1}N_{0}$                   & $0$                & $\L_{i+1}S$             & $\R_{i}S$              & $0$                  \\ \hline
$S$       & $U_{i}S$                       & $U_{i+1}S$                       & $\L_{i}N_{0}$       & $0$                  & $0$                   & $\R_{i+1}N_{0}$         \\ \hline
\end{tabular}}
\vspace{2mm}
\caption{The right multiplication maps for the six generators of $\M(\id) \oplus \M(\mathsf{X}_i)$.} \label{sixgentable}
\end{table}

\subsection{Our bimodule for a negative crossing}

Let $\sigma^{-1}_i$ be the elementary negative braid on $2n$ strands where the crossing is between the $i$ and $i+1$ strands. The complex of $\A$-bimodules $\M(\sigma_i^{-1})$ is given by the mapping cone 
\[ d^-: \M(\mathsf{X}_i) \to \M(\id) \]

\noindent
where $d^-$ is the map
\[
d^-(x_Z)=\begin{cases}
\begin{array}{lcl}
\iota_{t(Z)}&&\text{if}\ \{i+1\}\notin b(Z)\cup t(Z)\\
\mathcal{L}_i\iota_{t(Z)}&&\text{if}\ i+1\in b(Z)\ \text{and}\ i\in t(Z)\\
\mathcal{R}_i\iota_{t(Z)}&&\text{if}\ i\in b(Z)\ \text{and}\ i+1\in t(Z)\\
U_i \iota_{t(Z)}&&\text{if}\ \{i+1\}\subset b(Z)\cap t(Z).
\end{array}
\end{cases}
\]

As with the positive crossing, we will reinterpret this map in terms of the six generators $\{ N_{0}, S, N_{+},N_{-}, E, W \}$. The differential $d^-$ is given by
\[ 
\begin{split}
    & d^{-}(N_{+}) = N_{0} \\
    & d^{-}(N_{-}) = (U_{i} - U_{i+1})N_{0} \\
    & d^{-}(E) = \mathcal{L}_{i+1} S \\
    & d^{-}(W) = \mathcal{R}_i S \\
\end{split}
\]

\noindent
As a module, it is still given by $\M(\sigma_i) \oplus \M(\id)$, so the right multiplication is still described by Table \ref{sixgentable}.

\subsection{Our bimodule for the plat minimum}

Let $\vee(n)$ denote the plat minimum diagram consisting of $n$ cups. Each cup consists of two edges oriented out of a source -- suppose the edges are numbered $e_{1},...,e_{2n}$ so that $e_{i}$ lies on strand $i$. The bimodule $\M( \vee (n))$ is an $\A(0,0), \A(n,n)$-bimodule. 

The generators of $\M( \vee (n))$ correspond to $n$-component cycles $Z$ relative to the boundary of $\vee (n)$ union the bivalent vertices. For the $i$th minimum of $\vee(n)$, $Z$ contains either $e_{2i-1}$ or $e_{2i}$. There are $2^n$ such cycles.

The module $\M(\vee(n))$ is generated as a right module over $\A(n,n)$ by the $x_{Z}$ modulo the following relations:

\begin{enumerate}
    \item If $e_{2i-1} \in Z$ and $Z'=Z \setminus \{e_{2i-1} \} \cup \{e_{2i} \}$, then $x_{Z} \mathcal{R}_{2i-1} = x_{Z'}$.
    \item For all $i$ and $Z$, $X_{Z} R_{2i} =0$.
\end{enumerate}

\noindent
This bimodule is precisely the opposite bimodule to $\M(\wedge(n))$. In particular, it is also generated by the single generator $x_{Z_{odd}}$, where $Z_{odd}$ is the cycle which has the left edge at each minimum:

\[ Z_{odd} = \{e_1, e_3, ..., e_{2n-1} \} \]

\noindent
As in the plat maximum case, let $\mathbf{x}_{odd} = \{1,3,...,2n-1\}$.

\begin{lemma}
The module $\M(\vee(n))$ is isomorphic to $\iota_{\mathbf{x}_{odd}} \cdot \A(n,n)$.
\end{lemma}

\begin{proof}
The proof is the same as in the plat maximum case, and we leave it to the reader.
\end{proof}

\subsection{Relationship with Khovanov homology}

Let $b$ be a braid diagram on $2n$ strands with all strands oriented downwards. If $D$ has $m$ crossings, we can write $D$ as a product of elementary braids
\[ b = \prod_{i=1}^{m} \sigma_{j(i)} \]
with $j(i) \in \{1,-1,2,-2,...,n-1,-n+1\}$ and $\sigma_{-j}=\sigma_{j}^{-1}$.

Let $D=\mathsf{p}(b)$ be the plat closure of $D$. 

\begin{definition}

The module $\M(D)$ is given by tensor product
\[
\M(D) = \M(\vee(n)) \otimes_{\A} \M(\sigma_{j(1)}) \otimes_{\A} \M(\sigma_{j(2)})  \otimes_{\A} \cdot \cdot \cdot \otimes_{\A} \M(\sigma_{j(m)}) \otimes_{\A} \M(\wedge(n))
\]
where $\A=\A(n,n)$.
\end{definition}

Viewing $D$ as an oriented graph with $m$ 4-valent vertexes, $n$ bivalent sources, and $n$ bivalent sinks, let $e_{1},...,e_{2n+2m}$ denote the edges of $D$. We will define an action of $\mathbb{Q}[\mathcal{U}_1,...,\mathcal{U}_{2m+2n}]$ on $\M(D)$ as follows.

Between any two adjacent crossings, we can slice the diagram $D$ into two tangles $D=T_{1} \ast T_{2}$ where $T_{1}$ is a $(0,2n)$ tangle and $T_{2}$ is a $(2n,0)$ tangle and 
\[\M(T_{1}) \otimes_{\A} \M(T_2) = \M(D) \]

\noindent
Let $e_{k}$ denote the edge in $D$ corresponding to the $i$th strand where $D$ is sliced into $T_1$ and $T_2$. We define the action of $\mathcal{U}_k$ on $\M(D)$ by 
\[ \mathcal{U}_k (a \otimes b) = aU_i \otimes b = a \otimes U_i b \]

Let $w_{1}^+,...,w_{n}^+$ denote the bivalent maxima of $D$ (the sinks) ordered from left to right, and let $w_{1}^-,...,w_{n}^-$ denote the bivalent minima.

\begin{definition}

For each pair $w_{k}^+, w_k^-$, define \[L_{w_k^{\pm}} = U_{i_1(k)} + U_{i_2(k)} - U_{j_1(k)} - U_{j_2(k)} \quad \quad L'_{w_k^{\pm}} = U_{i_1(k)} + U_{i_2(k)} + U_{j_1(k)} + U_{j_2(k)},\]
where $e_{i_1(k)},e_{i_2(k)}$ are the edges adjacent to $w_{i}^-$ and $e_{j_1(k)},e_{j_2(k)}$ are the edges adjacent to $w_{i}^+$.

\end{definition}

Let $\mathsf{K}(D)$ denote the Koszul complex
\[ \mathsf{K}(D)=
\big{(}
\xymatrix{R\ar@<1ex>[r]^{L_{w_1^\pm}}&R\ar@<1ex>[l]^{L'_{w_1^{\pm}}}}\big{)}\otimes\big{(}
\xymatrix{R\ar@<1ex>[r]^{L_{w_2^\pm}}&R\ar@<1ex>[l]^{L'_{w_2^{\pm}}}}\big{)}\otimes...\otimes\big{(}
\xymatrix{R\ar@<1ex>[r]^{L_{w_n^\pm}}&R\ar@<1ex>[l]^{L'_{w_n^{\pm}}}}\big{)}\]

\begin{definition}

The complex $C_{1 \pm 1}(D)$ is defined to be the tensor product
\[C_{1 \pm 1}(D) = \M(D) \otimes \mathsf{K}(D). \]

\end{definition}

\begin{theorem}[\cite{alishahi2018link}]

Let $E_{k}(D)$ denote the spectral sequence induced by the cube filtration on $C_{1 \pm 1}(D)$. Then $E_{2}(D) \cong Kh(D)$, and the total homology $E_{\infty}(D)$ is a link invariant.

\end{theorem}

\section{Homotopy equivalences between the two types of bimodules}

In the previous two sections, we described two bimodules that could be ascribed to a plat tangle $T$, where a plat tangle is defined to be a horizontal slice of the plat closure of a braid: the Ozsv\'{a}th-Szab\'{o} $DA$ bimodule $\mathsf{OS}_{DA}(T)$ and our differential bimodule $\M(T)$. Our bimodule can also be viewed as a $DA$ bimodule $\M_{DA}(T)$ with  $\delta^{1}_{1}$ given by the differential, $\delta^{1}_{2}$ given by right multiplication, and $\delta^{1}_{i}=0$ for $i \ge 3$. Moreover, the $DA$ bimodules $\mathsf{OS}_{DA}(T)$ and $\M_{DA}(T)$ are defined over isomorphic algebras, where the isomorphism
\[ f: \A(n,k) \to A(n,2n-k) \]

\noindent
is given by taking the complement of each idempotent:
\[
\begin{split}
   & f(\iota_{\mathbf{x}}) = \mathbf{I_{\overline{x}}} \\
   & f(\mathcal{R}_{i}) = L_i \\
   & f(\mathcal{L}_{i}) = R_i \\
   & f(U_i) = U_{i}
\end{split}
\]

\begin{remark}

In our plat diagrams, we require the minima (resp. maxima) to have the same vertical coordinate, so that a plat tangle $T$ either includes all of the minima (resp. maxima) or none of them.

\end{remark}

\begin{theorem} \label{thm4.2}

The $DA$ bimodules $\mathsf{OS}_{DA}(T)$ and $\M_{DA}(m(T))$ are homotopy equivalent, where $m(T)$ is the mirror of $T$.

\end{theorem}

\noindent
This theorem will be proved by showing that it is true for $T=\sigma_{i}$, $T=\sigma_{i}^{-1}$, $T = \wedge(n)$, and $T=\vee(n)$. The result then follows from the fact the in both cases, gluing corresponds to box tensor product.

\subsection{The plat maximum and minimum}

We start with the easiest cases: the plat maximum $\wedge(n)$ and the plat minimum $\vee(n)$.

\begin{theorem} \label{thm4.3}

There are isomorphisms of type D and type A modules
\[
{}^{A}\mathsf{OS}_{D}(\wedge(n)) \cong {}^{A}\mathsf{M}_{D}(\wedge(n))
\]
\[
\mathsf{OS}_{A}(\vee(n))_{A} \cong \mathsf{M}_{A}(\vee(n))_{A}
\]
induced by the algebra isomorphism $f$.

\end{theorem}

\begin{proof}

This follows from the description of these modules in terms of $\mathbf{I}_{\mathbf{x}_{even}}$ and $\iota_{\mathbf{x}_{odd}}$, respectively, together with the fact that $f(\iota_{\mathbf{x}_{odd}})=\mathbf{I}_{\mathbf{x}_{even}}$.

\end{proof}

\subsection{The positive crossing $\sigma_i$}
By definition $\mathsf{M}(\sigma_i)$ is the mapping cone of the $\A \text{--}\A$-bimodule homomorphism
\[d^+:\M(\id) \to \mathsf{M}(\mathsf{X}_i) . \]

By Lemma \ref{freesingular}, the differential bimodule $\M(\sigma_{i})$ is a free left module generated by $\{N_{+}, N_{-}, E, W, N_{0}, S \}$ over the corresponding idempotent subalgebras of $\A$. Thus, it can be viewed as a $DA$ bimodule over $\A$ with these $6$ generators, where $\delta^{1}_{1}$ is the differential, $\delta^{1}_{2}$ is right multiplication, and there are no higher maps.

With respect to this basis, $\delta^{1}_{1}$ is given by 
\[ \delta^{1}_{1}(N_{0}) = N_{-} \]
\[ \delta^{1}_{1}(S) =\L_{i} \otimes W- \R_{i+1} \otimes E  . \]

\noindent
and the right multiplication $\delta^1_2$ is given by Table \ref{sixgentable}.

Recall that $\A$ is isomorphic to $A$ under the map $f$. Thus, we can define the $DA$ bimodule for $\sigma_i$ over $A$ by
\[ \delta^{1}_{1}(N_{0}) = N_{-} \]
\[ \delta^{1}_{1}(S) = R_{i} \otimes W-L_{i+1} \otimes E . \]

\noindent
and $\delta^{1}_{2}$ given by Table \ref{TableoverA}.

\begin{table}[!h]
\resizebox{\textwidth}{!}{%
\begin{tabular}{|l|l|l|l|l|l|l|} 
\hline 
 & $U_{i}$                        & $U_{i+1}$                        & $R_{i}$            & $R_{i+1}$              & $L_{i}$               & $L_{i+1}$              \\ \hline
$N_{+}$   & $N_{-}+U_{i+1}N_{+}$             & $-N_{-}+U_{i}N_{+}$            & $0$                & $E$                  & $W$                   & $0$                  \\ \hline
$N_{-}$   & $U_{i}N_{-}$                   & $U_{i+1}N_{-}$                   & $0$                & $R_{i+1}R_{i}W-U_{i+1}E$ & $-L_{i}L_{i+1}E+U_{i}W$ & $0$                  \\ \hline
$E$       & $R_{i+1}R_{i}W$                  & $(U_{i}+U_{i+1})E - R_{i+1}R_{i}W$ & $R_{i+1}R_{i}N_{+}$  & $0$                  & $0$                   & $U_{i}N_{+} - N_{-}$ \\ \hline
$W$       & $(U_{i}+U_{i+1})W - L_{i}L_{i+1}E$ & $L_{i}L_{i+1}E$                  & $N_{-}+U_{i+1}N_{+}$ & $0$                  & $0$                   & $L_{i}L_{i+1}N_{+}$    \\ \hline
$N_{0}$   & $U_{i}N_{0}$                   & $U_{i+1}N_{0}$                   & $0$                & $R_{i+1}S$             & $L_{i}S$              & $0$                  \\ \hline
$S$       & $U_{i}S$                       & $U_{i+1}S$                       & $R_{i}N_{0}$       & $0$                  & $0$                   & $L_{i+1}N_{0}$         \\ \hline
\end{tabular}}
\caption{} \label{TableoverA}
\end{table}

\begin{definition}
Let ${}^{A}\M_{DA}(\sigma_{i})_{A}$ denote this $DA$ bimodule.
\end{definition}

\begin{theorem} \label{thm4.5} There is a homotopy equivalence of $DA$ bimodules
\[ {}^{A}\M_{DA}(\sigma_{i})_{A} \simeq {}^{A} \mathsf{OS}_{DA}(\sigma^{-1}_i)_{A} .\] 

\end{theorem}

\begin{proof}

Recall that Ozsv\'{a}th and Szab\'{o} also define a $DD$-bimodule $\mathsf{OS}_{DD}(\sigma^{-1}_i)$ to a negative crossing which is related to $\mathsf{OS}_{DA}(\sigma^{-1}_i)$ by box tensor product with the identity $DD$-bimodule:
\[{}^{A, A'} \mathsf{OS}_{DD}(\sigma_i^{-1}) \cong {}^{A} \mathsf{OS}_{DA}(\sigma_i^{-1})_{A} \hspace{1.5mm} \boxtimes_{A} \hspace{0mm} {}^{A, A'} \mathsf{OS}_{DD}(\id) \]

Since ${}^{A, A'} \mathsf{OS}_{DD}(\id)$ is quasi-invertible, $\M_{DA}(\sigma_{i})$ is homotopy equivalent to $\mathsf{OS}_{DA}(\sigma_i^{-1})$ if and only if ${}^{A}\M_{DA}(\sigma_{i})_{A}  \boxtimes_{A} {}^{A, A'} \mathsf{OS}_{DD}(\id) $ is homotopy equivalent to ${}^{A, A'} \mathsf{OS}_{DA}(\sigma_i^{-1}) \boxtimes_{A} {}^{A, A'} \mathsf{OS}_{DD}(\id)$. Thus it is sufficient to show that 
\[ {}^{A}\M_{DA}(\sigma_{i})_A \boxtimes_{A} {}^{A, A'}\mathsf{OS}_{DD}(\id) \simeq {}^{A, A'} \mathsf{OS}_{DD}(\sigma_i^{-1}) .\]
The reason for tensoring with the identity $DD$-bimodule is that homological perturbation is simpler on the $DD$ side.

Recall that the differential on  ${}^{A, A'} \mathsf{OS}_{DD}(\id)$ is given by $\delta(x)=ax$, where 

\[  a = \sum_{i=1}^{2n} (L_{i} \otimes R'_{i} + R_{i} \otimes L'_{i}) - \sum_{i=1}^{2n} U_{i} \otimes E_{i} \in A \otimes A'. \]
The $DD$-bimodule \[{}^{A}\M_{DA}(\sigma_{i})_A\boxtimes {}^{A, A'} \mathsf{OS}_{DD}(\id) \]
is depicted as a table in Table \ref{deltapositive2}.

\begin{table}[!ht] 
\resizebox{\textwidth}{!}{%
{\renewcommand{\arraystretch}{2}%
\begin{tabular}{|l|l|l|l|l|l|l|}
\hline
\textbf{} & $N_+$                                          & $N_{-}$                                                & $E$                                                  & $W$                                                      & $N_0$                                           & $S$                                               \\ \hline
$N_+$   & \makecell{$-U_{i} \otimes E_{i+1}$ \\ $- U_{i+1} \otimes E_{i}$} & $1 \otimes (E_{i+1}-E_{i})$         & $1 \otimes L'_{i+1}$ &$1 \otimes R'_{i}$ & $0$                                               & $0$                                               \\ \hline
$N_{-}$   & $0$                      & \makecell{$-U_{i} \otimes E_{i}$ \\ $ - U_{i+1} \otimes E_{i+1}$}       & \makecell{$-U_{i+1} \otimes L'_{i+1} $\\ $-L_{i}L_{i+1}\otimes R'_{i}$}                               &  \makecell{$R_{i+1}R_{i}\otimes L'_{i+1} $ \\ $ + U_{i}\otimes R'_{i}$}                                         & $0$                                     & $0$                                               \\ \hline
$E$       &  \makecell{$R_{i+1}R_{i} \otimes L'_{i}$ \\ $ +U_{i} \otimes R'_{i+1}$}                            & $-1 \otimes R'_{i+1}$  & $-(U_{i}+U_{i+1}) \otimes E_{i+1}$                    & $R_{i+1}R_{i} \otimes (E_{i+1}-E_{i})$               & $0$                                               & $0$                               \\ \hline
$W$       & \makecell{ $U_{i+1} \otimes L'_{i} $\\ $ + L_{i}L_{i+1} \otimes R'_{i+1}$}                               & $1 \otimes L'_{i}$&     $L_{i}L_{i+1} \otimes (E_{i} - E_{i+1})$           & $-(U_{i}+U_{i+1}) \otimes E_{i}$                         & $0$                                               & $0$                                \\ \hline
$N_0$   & $0$                                              & $1\otimes 1$                                                    & $0$                                                  & $0$                                                      & \makecell{$-U_{i} \otimes E_{i} $\\ $ - U_{i+1} \otimes E_{i+1}$}  & \makecell{$-L_{i} \otimes R'_{i}$ \\ $ - R_{i+1} \otimes L'_{i+1}$} \\ \hline
$S$       & $0$                                              & $0$                                                    & $-L_{i+1}\otimes 1$                                                  & $R_i\otimes 1$                                                      & \makecell{$-L_{i+1} \otimes R'_{i+1} $ \\ $ -R_{i} \otimes L'_{i}$} & \makecell{$-U_{i} \otimes E_{i}$ \\ $- U_{i+1} \otimes E_{i+1}$}  \\ \hline
\end{tabular}}}
\vspace{3mm}
\caption{Table of entries for $\delta$ on ${}^{A}\M_{DA}(\sigma_{i})_A\boxtimes_{A} {}^{A, A'} \mathsf{OS}_{DD}(\id)$} \label{deltapositive2}
\end{table}

\begin{figure}
    \centering
\begin{tikzpicture}[scale=2.8]
    \node at (0,4) (N) {$N_{+}$} ;
    \node at (-2,2.5) (W) {$W$} ;
    \node at (2,2.5) (E) {$E$} ;
    \node at (0,0) (S) {$S$} ;
    \draw[->] (S) [bend left=10] to node[below,sloped] {\tiny{$R_{i} \otimes 1$}}  (W)  ;
    \draw[->] (W) [bend left=10] to node[above,sloped] {\tiny{$R_{i+1} \otimes L'_{i} L'_{i+1} + L_{i} \otimes U_{i}$}}  (S)  ;
    \draw[->] (E) [bend right=10] to node[above,sloped] {\tiny{$-R_{i+1} \otimes U_{i+1} - L_{i} \otimes R'_{i+1}R'_{i}$}}  (S)  ;
    \draw[->] (S)[bend right=10] to node[below,sloped] {\tiny{$-L_{i+1} \otimes 1$}} (E) ;
    \draw[->] (W) to node[below,sloped] {\tiny{$U_{i+1} \otimes L'_{i}+ L_{i}L_{i+1} \otimes R'_{i+1}$}} (N) ;
    \draw[->] (N)[bend right=20] to node[above,sloped] {\tiny{$1 \otimes R'_i$}} (W) ;
    \draw[->] (E) to node[below,sloped]{\tiny{$R_{i+1}R_i \otimes L'_i + U_i\otimes R'_{i+1}$}} (N) ;
    \draw[->] (N)[bend left=20] to node[above,sloped]{\tiny{$1 \otimes L'_{i+1}$}} (E) ;
    \draw[->] (N) [loop above, looseness=10] to node[above]{\tiny{$-U_i\otimes E_{i+1} - U_{i+1}\otimes E_{i}$}} (N);
     \draw[->] (S) [loop below, looseness=10] to node[below]{\tiny{$-U_i\otimes E_{i} - U_{i+1}\otimes E_{i+1}$}} (S);
    \draw[->] (W) [loop left, looseness=10] to node[above,sloped]{\tiny{$-(U_{i}+U_{i+1}) \otimes E_{i}$}} (W);
    \draw[->] (E) [loop right, looseness=10] to node[above,sloped]{\tiny{$-(U_i + U_{i+1}) \otimes E_{i+1}$}} (E);
    \draw[->] (E) [bend right=5] to node[above,pos=.3] {\tiny{$R_{i+1} R_i \otimes (E_{i+1} - E_i)$}} (W) ;
    \draw[->] (W) [bend right=5] to node[below,pos=.3] {\tiny{$L_{i}L_{i+1} \otimes (E_{i}-E_{i+1})$}} (E) ;
    \draw[->] (N) to node[above,sloped,pos=.66] {\tiny{$(R_{i+1} \otimes L'_{i+1} + L_{i} \otimes R'_{i})(1\otimes(E_{i+1}-E_{i}))$}} (S) ;
    \end{tikzpicture}

\caption{The $DD$ bimodule for ${}^{A}\M_{DA}(\sigma_{i})_A\boxtimes {}^{A, A'} \mathsf{OS}_{DD}(\id)$ after cancelation.} \label{PositiveCanceledDD}

\end{figure}
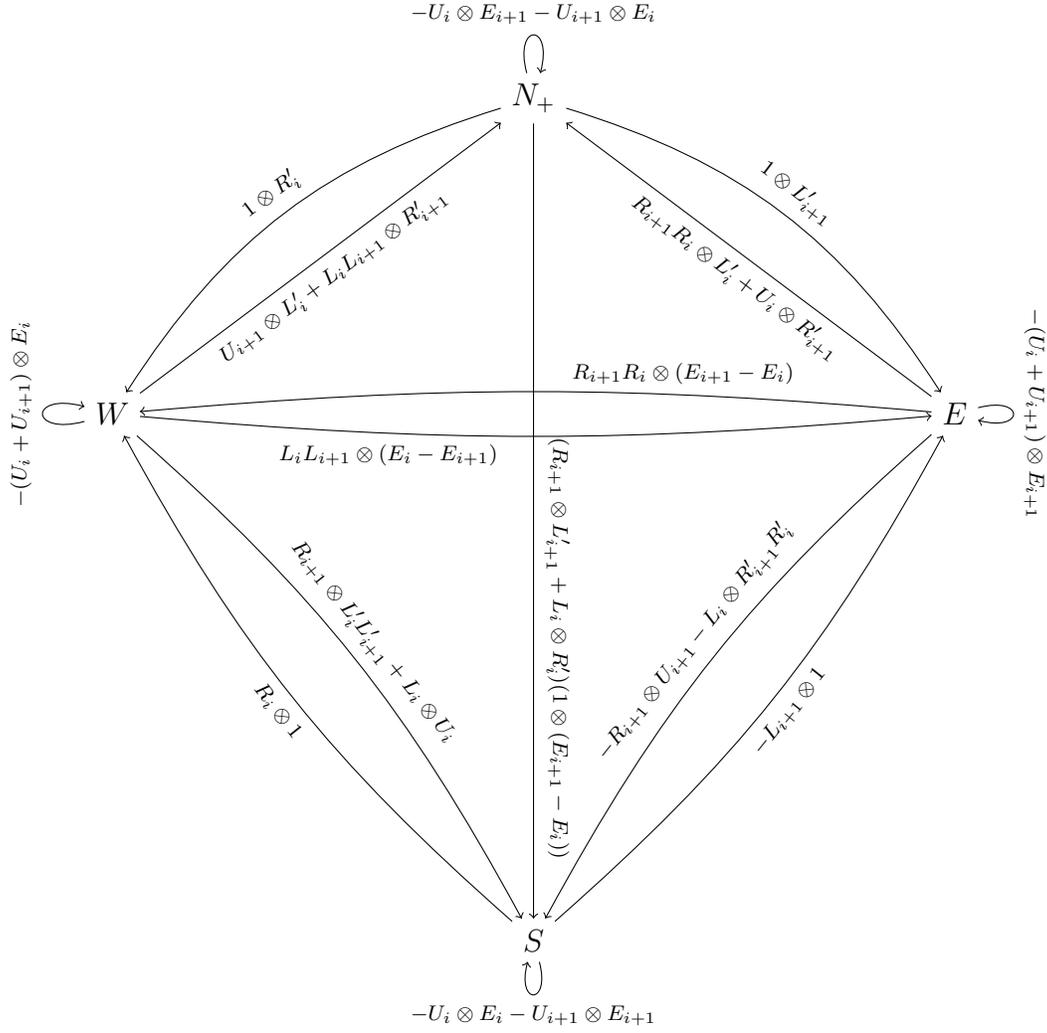

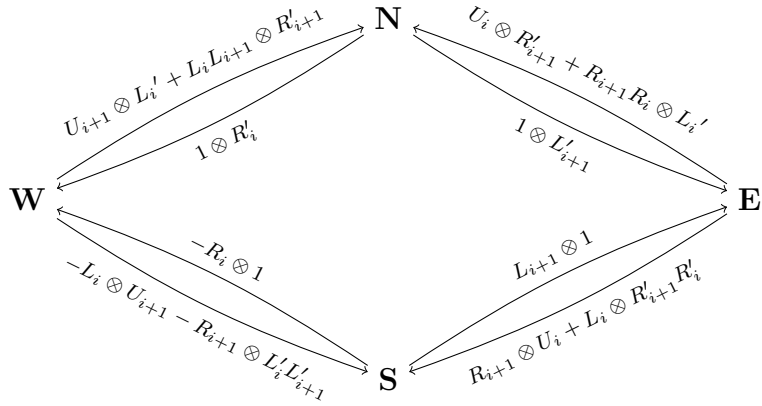
\begin{figure}
  \label{eq:NegativeCrossing}
\begin{tikzpicture}[scale=2.4]
    \node at (0,3) (N) {$\n$} ;
    \node at (-2,2) (W) {$\w$} ;
    \node at (2,2) (E) {$\e$} ;
    \node at (0,1) (S) {$\s$} ;
    \draw[->] (W) [bend left=7] to node[above,sloped] {\tiny{${U_{i+1}}\otimes{L_i}'+{L_{i} L_{i+1}}\otimes{R_{i+1}'}$}}  (N)  ;
    \draw[->] (N) [bend left=7] to node[below,sloped] {\tiny{${1}\otimes{R_{i}'}$}}  (W)  ;
    \draw[->] (N)[bend right=7] to node[below,sloped] {\tiny{${1}\otimes{L_{i+1}'}$}}  (E)  ;
    \draw[->] (E)[bend right=7] to node[above,sloped] {\tiny{${U_i}\otimes{R_{i+1}'} + {R_{i+1} R_{i}}\otimes{L_i}'$}} (N) ;
    \draw[->] (S)[bend right=7] to node[above,sloped] {\tiny{$-{R_i}\otimes{1}$}} (W) ;
    \draw[->] (W)[bend right=7] to node[below,sloped] {\tiny{$-{L_i}\otimes{U_{i+1}} - {R_{i+1}}\otimes{L_{i}' L_{i+1}'}$}} (S) ;
    \draw[->] (S)[bend left=7] to node[above,sloped]{\tiny{${L_{i+1}}\otimes{1}$}} (E) ;
    \draw[->] (E)[bend left=7] to node[below,sloped]{\tiny{${R_{i+1}}\otimes{U_{i}} + {L_i}\otimes{R_{i+1}' R_{i}'}$}} (S) ;
  \end{tikzpicture}
  
 \caption{The $DD$ bimodule ${}^{A, A'} \mathsf{OS}_{DD}(\sigma^{-1}_i)$. Each generator also has an arrow to itself with coefficient $-U_i\otimes E_{i+1} - U_{i+1}\otimes E_{i}$.} \label{OSDDNegative}
\end{figure}

Applying homological perturbation to the map $N_0 \mapsto (1 \otimes 1) \cdot N_{-}$ yields a $DD$ bimodule on the four generators $N_+, E, W, S$. The resulting differential is described in Figure \ref{PositiveCanceledDD}. Recall that the bimodule ${}^{A, A'} \mathsf{OS}_{DD}(\sigma_i)$ is given by Figure \ref{OSDDNegative}. Then the map given by 
\[
\begin{split}
   & \n  \mapsto N_{+} \\
&    \e  \mapsto E + R_{i+1} \otimes (E_{i} - E_{i+1}) \cdot S \\
&    \w  \mapsto W + L_{i} \otimes (E_{i}-E_{i+1}) \cdot S \\
&    \s  \mapsto S
\end{split}
\]
is an isomorphism of $DD$-bimodules, proving the theorem.

\end{proof}

\subsection{The negative crossing $\sigma_i^{-1}$} Let $\mathcal{A}=\mathcal{A}(n,k)$. The differential bimodule $\mathsf{M}(\sigma_i^{-1})$ is the mapping cone of the $\mathcal{A}-\mathcal{A}$-bimodule homomorphism 
\[d^-:\mathsf{M}(\mathsf{X}_i)\to\mathsf{M}(\mathrm{id}).\]

Similar to the positive crossing, the $\mathcal{A}-\mathcal{A}$-bimodule $\mathsf{M}(\sigma_i^{-1})$ can be viewed as a $DA$ bimodule over $\mathcal{A}$ with basis $\{N_+,N_-,E,W,N_0,S\}$. The differential is $\delta^1_1$, which with respect to this basis is given by 
\[\begin{split}
    & \delta^1_1(N_{+}) = N_{0} \\
    & \delta^1_1(N_{-}) = (U_{i} - U_{i+1})N_{0} \\
    & \delta_1^1(E) = \mathcal{L}_{i+1} S \\
    & \delta_1^1(W) = \mathcal{R}_i S. \\
\end{split}
\]
The map $\delta^1_2$ describes the right multiplication, which is given by Table \ref{sixgentable}. There are no higher maps.

As in the positive crossing case, under the isomorphism $f$ between $\mathcal{A}$ and $A$ we can view $\mathsf{M}(\sigma_i^-)$ as a DA bimodule over $A$ with $\delta^1_1$ given by 

\[\begin{split}
    & \delta^1_1(N_{+}) = N_{0} \\
    & \delta^1_1(N_{-}) = (U_{i} - U_{i+1})N_{0} \\
    & \delta_1^1(E) = R_{i+1} S \\
    & \delta_1^1(W) = L_i S. \\
\end{split}
\]
and $\delta^{1}_{2}$ given by Table \ref{TableoverA}.

\begin{definition}
Let $\lsup{A}\mathsf{M}_{DA}(\sigma_i^{-1})_{A}$ denote this $DA$ bimodule.
\end{definition}

\begin{theorem} \label{thm4.6} There is a homotopy equivalence of $DA$ bimodules
\[\lsup{A}\mathsf{M}_{DA}(\sigma_i^{-1})_A\simeq\lsup{A}\mathsf{OS}_{DA}(\sigma_i^{-1})_A\] 
\end{theorem}

\begin{proof}
As in the positive crossing case, we describe the explicit homological perturbation on the DD side. Specifically, we will show that 
\[\lsup{A,A'}\mathsf{OS}_{DD}(\sigma_i^{-1})\cong \lsup{A}\mathsf{OS}_{DA}(\sigma_i^{-1})_A\DT\lsup{A,A'}\mathsf{OS}_{DD}(\mathrm{id})\simeq \lsup{A}\mathsf{M}_{DA}(\sigma_i^{-1})_A\DT\lsup{A,A'}\mathsf{OS}_{DD}(\mathrm{id}),\]
and the theorem follows from the quasi-invertibility of $\lsup{A,A'}\mathsf{OS}_{DD}(\mathrm{id})$. The $DD$ bimodule 
\[\lsup{A}\mathsf{M}_{DA}(\sigma_i^{-1})_A\DT\lsup{A,A'}\mathsf{OS}_{DD}(\mathrm{id})\]
is given in Table \ref{deltanegative}.

\begin{table}[!ht] 
\resizebox{\textwidth}{!}{%
{\renewcommand{\arraystretch}{2}%
\begin{tabular}{|l|l|l|l|l|l|l|}
\hline
\textbf{} & $N_+$                                          & $N_{-}$                                                & $E$                                                  & $W$                                                      & $N_0$                                           & $S$                                               \\ \hline
$N_+$   & \makecell{$-U_{i} \otimes E_{i+1}$ \\ $- U_{i+1} \otimes E_{i}$} & $1 \otimes (E_{i+1}-E_{i})$         & $1 \otimes L'_{i+1}$ &$1 \otimes R'_{i}$ & $1\otimes 1$                                               & $0$                                               \\ \hline
$N_{-}$   & $0$                      & \makecell{$-U_{i} \otimes E_{i}$ \\ $ - U_{i+1} \otimes E_{i+1}$}       & \makecell{$-U_{i+1} \otimes L'_{i+1} $\\ $-L_{i}L_{i+1}\otimes R'_{i}$}                               &  \makecell{$R_{i+1}R_{i}\otimes L'_{i+1} $ \\ $ + U_{i}\otimes R'_{i}$}                                         & $(U_i-U_{i+1})\otimes 1$                                     & $0$                                               \\ \hline
$E$       &  \makecell{$R_{i+1}R_{i} \otimes L'_{i}$ \\ $ +U_{i} \otimes R'_{i+1}$}                            & $-1 \otimes R'_{i+1}$  & $-(U_{i}+U_{i+1}) \otimes E_{i+1}$                    & $R_{i+1}R_{i} \otimes (E_{i+1}-E_{i})$               & $0$                                               & $R_{i+1}\otimes 1$                               \\ \hline
$W$       & \makecell{ $U_{i+1} \otimes L'_{i} $\\ $ + L_{i}L_{i+1} \otimes R'_{i+1}$}                               & $1 \otimes L'_{i}$&     $L_{i}L_{i+1} \otimes (E_{i} - E_{i+1})$           & $-(U_{i}+U_{i+1}) \otimes E_{i}$                         & $0$                                               & $L_i\otimes 1$                                \\ \hline
$N_0$   & $0$                                              & $0$                                                    & $0$                                                  & $0$                                                      & \makecell{$-U_{i} \otimes E_{i} $\\ $ - U_{i+1} \otimes E_{i+1}$}  & \makecell{$-L_{i} \otimes R'_{i}$ \\ $ -R_{i+1} \otimes L'_{i+1}$} \\ \hline
$S$       & $0$                                              & $0$                                                    & $0$                                                  & $0$                                                      & \makecell{$-L_{i+1} \otimes R'_{i+1} $ \\ $ -R_{i} \otimes L'_{i}$} & \makecell{$-U_{i} \otimes E_{i}$ \\ $- U_{i+1} \otimes E_{i+1}$}  \\ \hline
\end{tabular}}}
\vspace{3mm}
\caption{Table of entries for $\delta$ on ${}^{A}\M_{DA}(\sigma_{i}^{-1})_A\boxtimes_{A} {}^{A, A'} \mathsf{OS}_{DD}(\id)$} \label{deltanegative}
\end{table}

Next, we apply homological perturbation to $N_+\to (1\otimes 1) \cdot N_0$ and get a $DD$ bimodule on four generators $N_-, E, W, S$. The resulting $\delta$ is given in Figure \ref{NegativeCanceledDD}. 

The following map gives the isomorphism with the $DD$ bimodule $\lsup{A,A'}\mathsf{OS}_{DD}(\sigma_i)$ given in Figure \ref{OSDDPositive}.  
\[
\begin{split}
&\n \to N_{-}\\
&\e \to E\\
&\w \to W\\
&\s \to S-L_{i+1}\otimes (E_{i+1}-E_{i})E-R_i\otimes (E_{i}-E_{i+1})W
\end{split}
\]
\end{proof}

\begin{figure}
    \centering
\begin{tikzpicture}[scale=2.8]
    \node at (0,4) (N) {$N_{-}$} ;
    \node at (-2,2.5) (W) {$W$} ;
    \node at (2,2.5) (E) {$E$} ;
    \node at (0,0) (S) {$S$} ;
    \draw[->] (S) [bend left=10] to node[below,sloped] {\tiny{$L_{i+1}\otimes R_{i+1}'R_{i}'+R_i\otimes U_i$}}  (W)  ;
    \draw[->] (W) [bend left=10] to node[above,sloped] {\tiny{$L_i\otimes 1$}}  (S)  ;
    \draw[->] (E) [bend right=10] to node[above,sloped] {\tiny{$R_{i+1}\otimes 1$}}  (S)  ;
    \draw[->] (S)[bend right=10] to node[below,sloped] {\tiny{$L_{i+1} \otimes U_{i+1}+R_i\otimes L_i'L_{i+1}'$}} (E) ;
    \draw[->] (W) to node[below,sloped] {\tiny{$1\otimes L_i'$}} (N) ;
    \draw[->] (N)[bend right=20] to node[above,sloped] {\tiny{$R_{i+1}R_i\otimes L_{i+1}'+U_{i+1}\otimes R_i'$}} (W) ;
    \draw[->] (E) to node[below,sloped]{\tiny{$-1\otimes R_{i+1}'$}} (N) ;
    \draw[->] (N)[bend left=20] to node[above,sloped]{\tiny{$-L_iL_{i+1}\otimes R_i'-U_i\otimes L_{i+1}'$}} (E) ;
    \draw[->] (N) [loop above, looseness=10] to node[above]{\tiny{$-U_i\otimes E_{i+1} - U_{i+1}\otimes E_{i}$}} (N);
        \draw[->] (S) [loop below, looseness=10] to node[below]{\tiny{$-U_i\otimes E_{i} -U_{i+1}\otimes E_{i+1}$}} (S);
    \draw[->] (W) [loop left, looseness=10] to node[above,sloped]{\tiny{$-(U_{i}+U_{i+1}) \otimes E_{i}$}} (W);
    \draw[->] (E) [loop right, looseness=10] to node[above,sloped]{\tiny{$-(U_i + U_i+1) \otimes E_{i+1}$}} (E);
    \draw[->] (E) [bend right=5] to node[above,pos=.3] {\tiny{$R_{i+1} R_i \otimes (E_{i+1} - E_i)$}} (W) ;
    \draw[->] (W) [bend right=5] to node[below,pos=.27] {\tiny{$L_{i}L_{i+1} \otimes (E_{i}-E_{i+1})$}} (E) ;
    \draw[->] (S) to node[above,sloped,pos=.33] {\tiny{$(R_i\otimes L_i'+L_{i+1}\otimes R_{i+1}')(1\otimes (E_{i+1}-E_{i}))$}} (N) ;
    \end{tikzpicture}

\caption{The $DD$ bimodule for ${}^{\A,\A}\M_{DA}(\sigma_{i})\boxtimes {}^{\A, \A'} \mathsf{OS}_{DD}(\id)$ after cancelation.} \label{NegativeCanceledDD}

\end{figure}

\begin{figure}
  \label{eq:NegativeCrossing}
\begin{tikzpicture}[scale=2.2]
    \node at (0,3) (N) {$N$} ;
    \node at (-2,2) (W) {$W$} ;
    \node at (2,2) (E) {$E$} ;
    \node at (0,1) (S) {$S$} ;
    \draw[->] (W) [bend right=7] to node[below,sloped] {\tiny{$1\otimes L_i'$}}  (N)  ;
    \draw[->] (N) [bend right=7] to node[above,sloped] {\tiny{$U_{i+1}\otimes R_i'+R_{i+1}R_i\otimes L_{i+1}'$}}  (W)  ;
    \draw[->] (N)[bend left=7] to node[above,sloped] {\tiny{$U_i\otimes L_{i+1}'+L_iL_{i+1}\otimes R_i'$}}  (E)  ;
    \draw[->] (E)[bend left=7] to node[below,sloped] {\tiny{$1\otimes R_{i+1}'$}} (N) ;
    \draw[->] (S)[bend left=7] to node[below,sloped] {\tiny{$-R_i\otimes U_{i+1}-L_{i+1}\otimes R_{i+1}'R_i'$}} (W) ;
    \draw[->] (W)[bend left=7] to node[above,sloped] {\tiny{$-L_i\otimes 1$}} (S) ;
    \draw[->] (S)[bend right=7] to node[below,sloped]{\tiny{$L_{i+1}\otimes U_i+R_i\otimes L_i'L_{i+1}'$}} (E) ;
    \draw[->] (E)[bend right=7] to node[above,sloped]{\tiny{$R_{i+1}\otimes 1$}} (S) ;
  \end{tikzpicture}
  
 \vspace{-2mm} 
 \caption{The $DD$ bimodule ${}^{A, A'} \mathsf{OS}_{DD}(\sigma_i)$. Each generator also has an arrow to itself with coefficient $-U_i\otimes E_{i+1} - U_{i+1}\otimes E_{i}$.} \label{OSDDPositive}.
\end{figure}
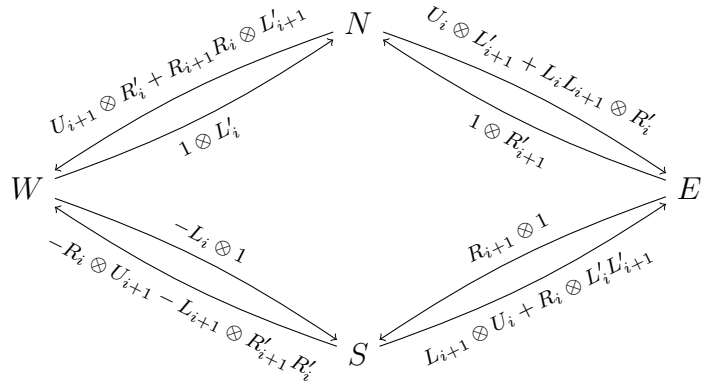

\begin{proof}[Proof of Theorem \ref{thm4.2}]

Theorem \ref{thm4.2} now follows from Theorems \ref{thm4.3}, \ref{thm4.5}, and \ref{thm4.6}, together with the fact that for both $DA$ bimodules, gluing corresponds to box-tensor product.

\end{proof}


\bibliographystyle{hamsalpha}
\bibliography{heegaardfloer}

\end{document}

%% file: defs.tex
\newcommand{\lsub}[2]{{}_{#1}#2}
\newcommand{\lsup}[2]{{}^{#1}\mskip-.6\thinmuskip#2}

\newcommand{\Alg}{\mathcal{A}}

\newcommand{\DT}{\boxtimes}

\newcommand{\y}{\mathbf{y}}

\theoremstyle{plain}

\numberwithin{equation}{section}
\newtheorem{theorem}[equation]{Theorem}

\newtheorem{lemma}[equation]{Lemma}
\newtheorem{corollary}[equation]{Corollary}

\newtheorem{conjecture}[equation]{Conjecture}

\newtheorem{convention}[equation]{Convention}
\newtheorem{definition}[equation]{Definition}

\theoremstyle{definition}

\theoremstyle{remark}

\newtheorem{remark}[equation]{Remark}

\definecolor{darkgreen}{rgb}{0,.25,0}
\definecolor{darkred}{rgb}{.25,0,0}
\definecolor{pink}{rgb}{1,.08,.575}

\makeatletter
\providecommand\@dotsep{5}
\def\listtodoname{List of Todos}
\def\listoftodos{\@starttoc{tdo}\listtodoname}
\makeatother

\newcommand{\RN}[1]{%
  \textup{\uppercase\expandafter{\romannumeral#1}}%
}

\newcommand{\rom}[1]{\uppercase\expandafter{\romannumeral #1\relax}}

\newcommand{\A}{\mathcal{A}}

\newcommand{\cupplus}{{\setbox0\hbox{\large$\cup$}\rlap{\hbox to \wd0{\hss\raisebox{3pt}{\tiny$+$}\hss}}\box0}}
\newcommand{\cupminus}{{\setbox0\hbox{\large$\cup$}\rlap{\hbox to \wd0{\hss\raisebox{3pt}{\tiny$-$}\hss}}\box0}}


%% file: figs/NCaps.pdf_tex
\begingroup%
  \makeatletter%
  \providecommand\color[2][]{%
    \errmessage{(Inkscape) Color is used for the text in Inkscape, but the package 'color.sty' is not loaded}%
    \renewcommand\color[2][]{}%
  }%
  \providecommand\transparent[1]{%
    \errmessage{(Inkscape) Transparency is used (non-zero) for the text in Inkscape, but the package 'transparent.sty' is not loaded}%
    \renewcommand\transparent[1]{}%
  }%
  \providecommand\rotatebox[2]{#2}%
  \ifx\svgwidth\undefined%
    \setlength{\unitlength}{572.22936929bp}%
    \ifx\svgscale\undefined%
      \relax%
    \else%
      \setlength{\unitlength}{\unitlength * \real{\svgscale}}%
    \fi%
  \else%
    \setlength{\unitlength}{\svgwidth}%
  \fi%
  \global\let\svgwidth\undefined%
  \global\let\svgscale\undefined%
  \makeatother%
  \begin{picture}(1,0.22608317)%
    \put(0,0){\includegraphics[width=\unitlength,page=1]{NCaps.pdf}}%
  \end{picture}%
\endgroup%

%% file: figs/NCups.pdf_tex
\begingroup%
  \makeatletter%
  \providecommand\color[2][]{%
    \errmessage{(Inkscape) Color is used for the text in Inkscape, but the package 'color.sty' is not loaded}%
    \renewcommand\color[2][]{}%
  }%
  \providecommand\transparent[1]{%
    \errmessage{(Inkscape) Transparency is used (non-zero) for the text in Inkscape, but the package 'transparent.sty' is not loaded}%
    \renewcommand\transparent[1]{}%
  }%
  \providecommand\rotatebox[2]{#2}%
  \ifx\svgwidth\undefined%
    \setlength{\unitlength}{572.22936929bp}%
    \ifx\svgscale\undefined%
      \relax%
    \else%
      \setlength{\unitlength}{\unitlength * \real{\svgscale}}%
    \fi%
  \else%
    \setlength{\unitlength}{\svgwidth}%
  \fi%
  \global\let\svgwidth\undefined%
  \global\let\svgscale\undefined%
  \makeatother%
  \begin{picture}(1,0.22608317)%
    \put(0,0){\includegraphics[width=\unitlength,page=1]{NCups.pdf}}%
  \end{picture}%
\endgroup%

%% file: figs/PositiveCrossingLabeled.pdf_tex
\begingroup%
  \makeatletter%
  \providecommand\color[2][]{%
    \errmessage{(Inkscape) Color is used for the text in Inkscape, but the package 'color.sty' is not loaded}%
    \renewcommand\color[2][]{}%
  }%
  \providecommand\transparent[1]{%
    \errmessage{(Inkscape) Transparency is used (non-zero) for the text in Inkscape, but the package 'transparent.sty' is not loaded}%
    \renewcommand\transparent[1]{}%
  }%
  \providecommand\rotatebox[2]{#2}%
  \ifx\svgwidth\undefined%
    \setlength{\unitlength}{218.49699637bp}%
    \ifx\svgscale\undefined%
      \relax%
    \else%
      \setlength{\unitlength}{\unitlength * \real{\svgscale}}%
    \fi%
  \else%
    \setlength{\unitlength}{\svgwidth}%
  \fi%
  \global\let\svgwidth\undefined%
  \global\let\svgscale\undefined%
  \makeatother%
  \begin{picture}(1,0.86819036)%
    \put(0,0){\includegraphics[width=\unitlength,page=1]{PositiveCrossingLabeled.pdf}}%
    \put(0.63329246,0.63702511){\color[rgb]{0,0,0}\makebox(0,0)[lb]{\smash{}}}%
    \put(0.48684585,0.27015493){\color[rgb]{0,0,0}\makebox(0,0)[lb]{\smash{$\mathbf{S}$}}}%
    \put(0,0){\includegraphics[width=\unitlength,page=2]{PositiveCrossingLabeled.pdf}}%
    \put(0.69920575,0.40196452){\color[rgb]{0,0,0}\makebox(0,0)[lb]{\smash{$\mathbf{E}$}}}%
    \put(0.46487759,0.59026389){\color[rgb]{0,0,0}\makebox(0,0)[lb]{\smash{$\mathbf{N}$}}}%
    \put(0.25984044,0.41661003){\color[rgb]{0,0,0}\makebox(0,0)[lb]{\smash{$\mathbf{W}$}}}%
  \end{picture}%
\endgroup%

%% file: figs/NegativeCrossingLabeled.pdf_tex
\begingroup%
  \makeatletter%
  \providecommand\color[2][]{%
    \errmessage{(Inkscape) Color is used for the text in Inkscape, but the package 'color.sty' is not loaded}%
    \renewcommand\color[2][]{}%
  }%
  \providecommand\transparent[1]{%
    \errmessage{(Inkscape) Transparency is used (non-zero) for the text in Inkscape, but the package 'transparent.sty' is not loaded}%
    \renewcommand\transparent[1]{}%
  }%
  \providecommand\rotatebox[2]{#2}%
  \ifx\svgwidth\undefined%
    \setlength{\unitlength}{218.49697047bp}%
    \ifx\svgscale\undefined%
      \relax%
    \else%
      \setlength{\unitlength}{\unitlength * \real{\svgscale}}%
    \fi%
  \else%
    \setlength{\unitlength}{\svgwidth}%
  \fi%
  \global\let\svgwidth\undefined%
  \global\let\svgscale\undefined%
  \makeatother%
  \begin{picture}(1,0.86819047)%
    \put(0,0){\includegraphics[width=\unitlength,page=1]{NegativeCrossingLabeled.pdf}}%
    \put(0.63329248,0.63702523){\color[rgb]{0,0,0}\makebox(0,0)[lb]{\smash{}}}%
    \put(0.47220034,0.24689595){\color[rgb]{0,0,0}\makebox(0,0)[lb]{\smash{$\mathbf{S}$}}}%
    \put(0,0){\includegraphics[width=\unitlength,page=2]{NegativeCrossingLabeled.pdf}}%
    \put(0.69920577,0.43728761){\color[rgb]{0,0,0}\makebox(0,0)[lb]{\smash{$\mathbf{E}$}}}%
    \put(0.47220034,0.57432771){\color[rgb]{0,0,0}\makebox(0,0)[lb]{\smash{$\mathbf{N}$}}}%
    \put(0.25984041,0.42264209){\color[rgb]{0,0,0}\makebox(0,0)[lb]{\smash{$\mathbf{W}$}}}%
  \end{picture}%
\endgroup%

%% file: figs/FigureEightCurveLabeled.pdf_tex
\begingroup%
  \makeatletter%
  \providecommand\color[2][]{%
    \errmessage{(Inkscape) Color is used for the text in Inkscape, but the package 'color.sty' is not loaded}%
    \renewcommand\color[2][]{}%
  }%
  \providecommand\transparent[1]{%
    \errmessage{(Inkscape) Transparency is used (non-zero) for the text in Inkscape, but the package 'transparent.sty' is not loaded}%
    \renewcommand\transparent[1]{}%
  }%
  \providecommand\rotatebox[2]{#2}%
  \ifx\svgwidth\undefined%
    \setlength{\unitlength}{218.49697047bp}%
    \ifx\svgscale\undefined%
      \relax%
    \else%
      \setlength{\unitlength}{\unitlength * \real{\svgscale}}%
    \fi%
  \else%
    \setlength{\unitlength}{\svgwidth}%
  \fi%
  \global\let\svgwidth\undefined%
  \global\let\svgscale\undefined%
  \makeatother%
  \begin{picture}(1,0.86819047)%
    \put(0,0){\includegraphics[width=\unitlength,page=1]{FigureEightCurveLabeled.pdf}}%
    \put(0.29332323,0.62785324){\color[rgb]{0,0,0}\makebox(0,0)[lb]{\smash{$N_{+}$}}}%
    \put(0.63329248,0.63702523){\color[rgb]{0,0,0}\makebox(0,0)[lb]{\smash{}}}%
    \put(0.63017,0.62785324){\color[rgb]{0,0,0}\makebox(0,0)[lb]{\smash{$N_{-}$}}}%
    \put(0.79859339,0.48139812){\color[rgb]{0,0,0}\makebox(0,0)[lb]{\smash{$E$}}}%
    \put(0.15778688,0.46806025){\color[rgb]{0,0,0}\makebox(0,0)[lb]{\smash{$W$}}}%
  \end{picture}%
\endgroup%

%% file: figs/OrientedSmoothingLabeled.pdf_tex
\begingroup%
  \makeatletter%
  \providecommand\color[2][]{%
    \errmessage{(Inkscape) Color is used for the text in Inkscape, but the package 'color.sty' is not loaded}%
    \renewcommand\color[2][]{}%
  }%
  \providecommand\transparent[1]{%
    \errmessage{(Inkscape) Transparency is used (non-zero) for the text in Inkscape, but the package 'transparent.sty' is not loaded}%
    \renewcommand\transparent[1]{}%
  }%
  \providecommand\rotatebox[2]{#2}%
  \ifx\svgwidth\undefined%
    \setlength{\unitlength}{218.49697047bp}%
    \ifx\svgscale\undefined%
      \relax%
    \else%
      \setlength{\unitlength}{\unitlength * \real{\svgscale}}%
    \fi%
  \else%
    \setlength{\unitlength}{\svgwidth}%
  \fi%
  \global\let\svgwidth\undefined%
  \global\let\svgscale\undefined%
  \makeatother%
  \begin{picture}(1,0.86819047)%
    \put(0,0){\includegraphics[width=\unitlength,page=1]{OrientedSmoothingLabeled.pdf}}%
    \put(0.63329248,0.63702523){\color[rgb]{0,0,0}\makebox(0,0)[lb]{\smash{}}}%
    \put(0.38720403,0.22989671){\color[rgb]{0,0,0}\makebox(0,0)[lb]{\smash{$S$}}}%
    \put(0.35922803,0.59856221){\color[rgb]{0,0,0}\makebox(0,0)[lb]{\smash{$N_{0}$}}}%
  \end{picture}%
\endgroup%

%% file: figs/NegativeCrossingConeImmersed.pdf_tex
\begingroup%
  \makeatletter%
  \providecommand\color[2][]{%
    \errmessage{(Inkscape) Color is used for the text in Inkscape, but the package 'color.sty' is not loaded}%
    \renewcommand\color[2][]{}%
  }%
  \providecommand\transparent[1]{%
    \errmessage{(Inkscape) Transparency is used (non-zero) for the text in Inkscape, but the package 'transparent.sty' is not loaded}%
    \renewcommand\transparent[1]{}%
  }%
  \providecommand\rotatebox[2]{#2}%
  \ifx\svgwidth\undefined%
    \setlength{\unitlength}{218.49697047bp}%
    \ifx\svgscale\undefined%
      \relax%
    \else%
      \setlength{\unitlength}{\unitlength * \real{\svgscale}}%
    \fi%
  \else%
    \setlength{\unitlength}{\svgwidth}%
  \fi%
  \global\let\svgwidth\undefined%
  \global\let\svgscale\undefined%
  \makeatother%
  \begin{picture}(1,0.86819047)%
    \put(0,0){\includegraphics[width=\unitlength,page=1]{NegativeCrossingConeImmersed.pdf}}%
    \put(0.63329248,0.63702523){\color[rgb]{0,0,0}\makebox(0,0)[lb]{\smash{}}}%
    \put(0.45015765,0.2449723){\color[rgb]{0,0,0}\makebox(0,0)[lb]{\smash{$S$}}}%
    \put(0,0){\includegraphics[width=\unitlength,page=2]{NegativeCrossingConeImmersed.pdf}}%
    \put(0.79793225,0.48253168){\color[rgb]{0,0,0}\makebox(0,0)[lb]{\smash{$E$}}}%
    \put(0.46487758,0.57641997){\color[rgb]{0,0,0}\makebox(0,0)[lb]{\smash{$N_{0}$}}}%
    \put(0.25003314,0.42578043){\color[rgb]{0,0,0}\makebox(0,0)[lb]{\smash{$W$}}}%
    \put(0,0){\includegraphics[width=\unitlength,page=3]{NegativeCrossingConeImmersed.pdf}}%
    \put(0.32574522,0.60571099){\color[rgb]{0,0,0}\makebox(0,0)[lb]{\smash{$N_{+}$}}}%
    \put(0,0){\includegraphics[width=\unitlength,page=4]{NegativeCrossingConeImmersed.pdf}}%
    \put(0.64860025,0.64376316){\color[rgb]{0,0,0}\makebox(0,0)[lb]{\smash{$N_{-}$}}}%
    \put(0,0){\includegraphics[width=\unitlength,page=5]{NegativeCrossingConeImmersed.pdf}}%
  \end{picture}%
\endgroup%

%% file: figs/NegativeCrossingConeEmbedded.pdf_tex
\begingroup%
  \makeatletter%
  \providecommand\color[2][]{%
    \errmessage{(Inkscape) Color is used for the text in Inkscape, but the package 'color.sty' is not loaded}%
    \renewcommand\color[2][]{}%
  }%
  \providecommand\transparent[1]{%
    \errmessage{(Inkscape) Transparency is used (non-zero) for the text in Inkscape, but the package 'transparent.sty' is not loaded}%
    \renewcommand\transparent[1]{}%
  }%
  \providecommand\rotatebox[2]{#2}%
  \ifx\svgwidth\undefined%
    \setlength{\unitlength}{218.49697047bp}%
    \ifx\svgscale\undefined%
      \relax%
    \else%
      \setlength{\unitlength}{\unitlength * \real{\svgscale}}%
    \fi%
  \else%
    \setlength{\unitlength}{\svgwidth}%
  \fi%
  \global\let\svgwidth\undefined%
  \global\let\svgscale\undefined%
  \makeatother%
  \begin{picture}(1,0.86819047)%
    \put(0,0){\includegraphics[width=\unitlength,page=1]{NegativeCrossingConeEmbedded.pdf}}%
    \put(0.63329248,0.63702523){\color[rgb]{0,0,0}\makebox(0,0)[lb]{\smash{}}}%
    \put(0.43668906,0.24418768){\color[rgb]{0,0,0}\makebox(0,0)[lb]{\smash{$S$}}}%
    \put(0,0){\includegraphics[width=\unitlength,page=2]{NegativeCrossingConeEmbedded.pdf}}%
    \put(0.79793225,0.47913188){\color[rgb]{0,0,0}\makebox(0,0)[lb]{\smash{$E$}}}%
    \put(0.51288434,0.58466734){\color[rgb]{0,0,0}\makebox(0,0)[lb]{\smash{$N_{0}$}}}%
    \put(0.25373167,0.39996497){\color[rgb]{0,0,0}\makebox(0,0)[lb]{\smash{$W$}}}%
    \put(0,0){\includegraphics[width=\unitlength,page=3]{NegativeCrossingConeEmbedded.pdf}}%
    \put(0.32019739,0.60571099){\color[rgb]{0,0,0}\makebox(0,0)[lb]{\smash{$N_{+}$}}}%
    \put(0,0){\includegraphics[width=\unitlength,page=4]{NegativeCrossingConeEmbedded.pdf}}%
    \put(0.65507269,0.64931098){\color[rgb]{0,0,0}\makebox(0,0)[lb]{\smash{$N_{-}$}}}%
    \put(0,0){\includegraphics[width=\unitlength,page=5]{NegativeCrossingConeEmbedded.pdf}}%
  \end{picture}%
\endgroup%

%% file: figs/IdempotentExampleAD.pdf_tex
\begingroup%
  \makeatletter%
  \providecommand\color[2][]{%
    \errmessage{(Inkscape) Color is used for the text in Inkscape, but the package 'color.sty' is not loaded}%
    \renewcommand\color[2][]{}%
  }%
  \providecommand\transparent[1]{%
    \errmessage{(Inkscape) Transparency is used (non-zero) for the text in Inkscape, but the package 'transparent.sty' is not loaded}%
    \renewcommand\transparent[1]{}%
  }%
  \providecommand\rotatebox[2]{#2}%
  \ifx\svgwidth\undefined%
    \setlength{\unitlength}{227.131024bp}%
    \ifx\svgscale\undefined%
      \relax%
    \else%
      \setlength{\unitlength}{\unitlength * \real{\svgscale}}%
    \fi%
  \else%
    \setlength{\unitlength}{\svgwidth}%
  \fi%
  \global\let\svgwidth\undefined%
  \global\let\svgscale\undefined%
  \makeatother%
  \begin{picture}(1,0.19468447)%
    \put(0,0){\includegraphics[width=\unitlength,page=1]{IdempotentExampleAD.pdf}}%
    \put(0.08161843,0.00647083){\color[rgb]{0,0,0}\makebox(0,0)[lb]{\smash{$1$}}}%
    \put(0.29295019,0.00647083){\color[rgb]{0,0,0}\makebox(0,0)[lb]{\smash{$2$}}}%
    \put(0.50428196,0.00647083){\color[rgb]{0,0,0}\makebox(0,0)[lb]{\smash{$3$}}}%
    \put(0.71561372,0.00647083){\color[rgb]{0,0,0}\makebox(0,0)[lb]{\smash{$4$}}}%
    \put(0,0){\includegraphics[width=\unitlength,page=2]{IdempotentExampleAD.pdf}}%
  \end{picture}%
\endgroup%

%% file: figs/IdempotentExampleOS.pdf_tex
\begingroup%
  \makeatletter%
  \providecommand\color[2][]{%
    \errmessage{(Inkscape) Color is used for the text in Inkscape, but the package 'color.sty' is not loaded}%
    \renewcommand\color[2][]{}%
  }%
  \providecommand\transparent[1]{%
    \errmessage{(Inkscape) Transparency is used (non-zero) for the text in Inkscape, but the package 'transparent.sty' is not loaded}%
    \renewcommand\transparent[1]{}%
  }%
  \providecommand\rotatebox[2]{#2}%
  \ifx\svgwidth\undefined%
    \setlength{\unitlength}{228.80467882bp}%
    \ifx\svgscale\undefined%
      \relax%
    \else%
      \setlength{\unitlength}{\unitlength * \real{\svgscale}}%
    \fi%
  \else%
    \setlength{\unitlength}{\svgwidth}%
  \fi%
  \global\let\svgwidth\undefined%
  \global\let\svgscale\undefined%
  \makeatother%
  \begin{picture}(1,0.19326039)%
    \put(0,0){\includegraphics[width=\unitlength,page=1]{IdempotentExampleOS.pdf}}%
    \put(0.29080733,0.0064235){\color[rgb]{0,0,0}\makebox(0,0)[lb]{\smash{$1$}}}%
    \put(0.50059325,0.0064235){\color[rgb]{0,0,0}\makebox(0,0)[lb]{\smash{$2$}}}%
    \put(0.71037917,0.0064235){\color[rgb]{0,0,0}\makebox(0,0)[lb]{\smash{$3$}}}%
    \put(0.92016509,0.0064235){\color[rgb]{0,0,0}\makebox(0,0)[lb]{\smash{$4$}}}%
    \put(0,0){\includegraphics[width=\unitlength,page=2]{IdempotentExampleOS.pdf}}%
    \put(0.08102141,0.0064235){\color[rgb]{0,0,0}\makebox(0,0)[lb]{\smash{$0$}}}%
  \end{picture}%
\endgroup%

%% file: figs/OneCap.pdf_tex
\begingroup%
  \makeatletter%
  \providecommand\color[2][]{%
    \errmessage{(Inkscape) Color is used for the text in Inkscape, but the package 'color.sty' is not loaded}%
    \renewcommand\color[2][]{}%
  }%
  \providecommand\transparent[1]{%
    \errmessage{(Inkscape) Transparency is used (non-zero) for the text in Inkscape, but the package 'transparent.sty' is not loaded}%
    \renewcommand\transparent[1]{}%
  }%
  \providecommand\rotatebox[2]{#2}%
  \ifx\svgwidth\undefined%
    \setlength{\unitlength}{572.22936929bp}%
    \ifx\svgscale\undefined%
      \relax%
    \else%
      \setlength{\unitlength}{\unitlength * \real{\svgscale}}%
    \fi%
  \else%
    \setlength{\unitlength}{\svgwidth}%
  \fi%
  \global\let\svgwidth\undefined%
  \global\let\svgscale\undefined%
  \makeatother%
  \begin{picture}(1,0.28911559)%
    \put(0,0){\includegraphics[width=\unitlength,page=1]{OneCap.pdf}}%
    \put(0.06185413,0.00154105){\color[rgb]{0,0,0}\makebox(0,0)[lb]{\smash{$1$}}}%
    \put(0.29952108,0.00154105){\color[rgb]{0,0,0}\makebox(0,0)[lb]{\smash{$c-1$}}}%
    \put(0.40577218,0.00154105){\color[rgb]{0,0,0}\makebox(0,0)[lb]{\smash{$c$}}}%
    \put(0.54278017,0.00154105){\color[rgb]{0,0,0}\makebox(0,0)[lb]{\smash{$c+1$}}}%
    \put(0.63505087,0.00154105){\color[rgb]{0,0,0}\makebox(0,0)[lb]{\smash{$c+2$}}}%
    \put(0.88390213,0.00154105){\color[rgb]{0,0,0}\makebox(0,0)[lb]{\smash{$2n+2$}}}%
    \put(0.06185413,0.28114921){\color[rgb]{0,0,0}\makebox(0,0)[lb]{\smash{$1$}}}%
    \put(0.29952108,0.28114921){\color[rgb]{0,0,0}\makebox(0,0)[lb]{\smash{$c-1$}}}%
    \put(0.65462344,0.28114921){\color[rgb]{0,0,0}\makebox(0,0)[lb]{\smash{$c$}}}%
    \put(0.89229037,0.28114921){\color[rgb]{0,0,0}\makebox(0,0)[lb]{\smash{$2n$}}}%
  \end{picture}%
\endgroup%

%% file: figs/NCapsGenerator.pdf_tex
\begingroup%
  \makeatletter%
  \providecommand\color[2][]{%
    \errmessage{(Inkscape) Color is used for the text in Inkscape, but the package 'color.sty' is not loaded}%
    \renewcommand\color[2][]{}%
  }%
  \providecommand\transparent[1]{%
    \errmessage{(Inkscape) Transparency is used (non-zero) for the text in Inkscape, but the package 'transparent.sty' is not loaded}%
    \renewcommand\transparent[1]{}%
  }%
  \providecommand\rotatebox[2]{#2}%
  \ifx\svgwidth\undefined%
    \setlength{\unitlength}{576.17555023bp}%
    \ifx\svgscale\undefined%
      \relax%
    \else%
      \setlength{\unitlength}{\unitlength * \real{\svgscale}}%
    \fi%
  \else%
    \setlength{\unitlength}{\svgwidth}%
  \fi%
  \global\let\svgwidth\undefined%
  \global\let\svgscale\undefined%
  \makeatother%
  \begin{picture}(1,0.15932461)%
    \put(0,0){\includegraphics[width=\unitlength,page=1]{NCapsGenerator.pdf}}%
  \end{picture}%
\endgroup%

%% file: figs/OrderedCups.pdf_tex
\begingroup%
  \makeatletter%
  \providecommand\color[2][]{%
    \errmessage{(Inkscape) Color is used for the text in Inkscape, but the package 'color.sty' is not loaded}%
    \renewcommand\color[2][]{}%
  }%
  \providecommand\transparent[1]{%
    \errmessage{(Inkscape) Transparency is used (non-zero) for the text in Inkscape, but the package 'transparent.sty' is not loaded}%
    \renewcommand\transparent[1]{}%
  }%
  \providecommand\rotatebox[2]{#2}%
  \ifx\svgwidth\undefined%
    \setlength{\unitlength}{572.22936929bp}%
    \ifx\svgscale\undefined%
      \relax%
    \else%
      \setlength{\unitlength}{\unitlength * \real{\svgscale}}%
    \fi%
  \else%
    \setlength{\unitlength}{\svgwidth}%
  \fi%
  \global\let\svgwidth\undefined%
  \global\let\svgscale\undefined%
  \makeatother%
  \begin{picture}(1,0.82147356)%
    \put(0,0){\includegraphics[width=\unitlength,page=1]{OrderedCups.pdf}}%
  \end{picture}%
\endgroup%

%% file: figs/R1OS.pdf_tex
\begingroup%
  \makeatletter%
  \providecommand\color[2][]{%
    \errmessage{(Inkscape) Color is used for the text in Inkscape, but the package 'color.sty' is not loaded}%
    \renewcommand\color[2][]{}%
  }%
  \providecommand\transparent[1]{%
    \errmessage{(Inkscape) Transparency is used (non-zero) for the text in Inkscape, but the package 'transparent.sty' is not loaded}%
    \renewcommand\transparent[1]{}%
  }%
  \providecommand\rotatebox[2]{#2}%
  \ifx\svgwidth\undefined%
    \setlength{\unitlength}{227.535088bp}%
    \ifx\svgscale\undefined%
      \relax%
    \else%
      \setlength{\unitlength}{\unitlength * \real{\svgscale}}%
    \fi%
  \else%
    \setlength{\unitlength}{\svgwidth}%
  \fi%
  \global\let\svgwidth\undefined%
  \global\let\svgscale\undefined%
  \makeatother%
  \begin{picture}(1,0.52294507)%
    \put(0,0){\includegraphics[width=\unitlength,page=1]{R1OS.pdf}}%
  \end{picture}%
\endgroup%

%% file: figs/R1AD.pdf_tex
\begingroup%
  \makeatletter%
  \providecommand\color[2][]{%
    \errmessage{(Inkscape) Color is used for the text in Inkscape, but the package 'color.sty' is not loaded}%
    \renewcommand\color[2][]{}%
  }%
  \providecommand\transparent[1]{%
    \errmessage{(Inkscape) Transparency is used (non-zero) for the text in Inkscape, but the package 'transparent.sty' is not loaded}%
    \renewcommand\transparent[1]{}%
  }%
  \providecommand\rotatebox[2]{#2}%
  \ifx\svgwidth\undefined%
    \setlength{\unitlength}{227.535088bp}%
    \ifx\svgscale\undefined%
      \relax%
    \else%
      \setlength{\unitlength}{\unitlength * \real{\svgscale}}%
    \fi%
  \else%
    \setlength{\unitlength}{\svgwidth}%
  \fi%
  \global\let\svgwidth\undefined%
  \global\let\svgscale\undefined%
  \makeatother%
  \begin{picture}(1,0.51854779)%
    \put(0,0){\includegraphics[width=\unitlength,page=1]{R1AD.pdf}}%
  \end{picture}%
\endgroup%

%% file: figs/AlgMult1.pdf_tex
\begingroup%
  \makeatletter%
  \providecommand\color[2][]{%
    \errmessage{(Inkscape) Color is used for the text in Inkscape, but the package 'color.sty' is not loaded}%
    \renewcommand\color[2][]{}%
  }%
  \providecommand\transparent[1]{%
    \errmessage{(Inkscape) Transparency is used (non-zero) for the text in Inkscape, but the package 'transparent.sty' is not loaded}%
    \renewcommand\transparent[1]{}%
  }%
  \providecommand\rotatebox[2]{#2}%
  \ifx\svgwidth\undefined%
    \setlength{\unitlength}{936.00291748bp}%
    \ifx\svgscale\undefined%
      \relax%
    \else%
      \setlength{\unitlength}{\unitlength * \real{\svgscale}}%
    \fi%
  \else%
    \setlength{\unitlength}{\svgwidth}%
  \fi%
  \global\let\svgwidth\undefined%
  \global\let\svgscale\undefined%
  \makeatother%
  \begin{picture}(1,1.17699674)%
    \put(0,0){\includegraphics[width=\unitlength,page=1]{AlgMult1.pdf}}%
    \put(0.49882098,1.08450449){\color[rgb]{0,0,0}\makebox(0,0)[lb]{\smash{$=$}}}%
    \put(0,0){\includegraphics[width=\unitlength,page=2]{AlgMult1.pdf}}%
    \put(0.49882098,0.75800977){\color[rgb]{0,0,0}\makebox(0,0)[lb]{\smash{$=$}}}%
    \put(0.82360629,0.75971917){\color[rgb]{0,0,0}\makebox(0,0)[lb]{\smash{$=0$}}}%
    \put(0,0){\includegraphics[width=\unitlength,page=3]{AlgMult1.pdf}}%
    \put(0.34155651,0.39732714){\color[rgb]{0,0,0}\makebox(0,0)[lb]{\smash{$=$}}}%
    \put(0.14070676,0.27572572){\color[rgb]{0,0,0}\makebox(0,0)[lb]{\smash{{\tiny $i$}}}}%
    \put(0.2159202,0.27572572){\color[rgb]{0,0,0}\makebox(0,0)[lb]{\smash{\tiny $i+1$}}}%
    \put(0.46549208,0.27572572){\color[rgb]{0,0,0}\makebox(0,0)[lb]{\smash{\tiny $i$}}}%
    \put(0.54070552,0.27572572){\color[rgb]{0,0,0}\makebox(0,0)[lb]{\smash{\tiny $i+1$}}}%
    \put(-0.00088683,0.39732271){\color[rgb]{0,0,0}\makebox(0,0)[lb]{\smash{$U_{i} \hspace{1mm} \cdot$}}}%
    \put(0,0){\includegraphics[width=\unitlength,page=4]{AlgMult1.pdf}}%
    \put(0.65608544,0.39732714){\color[rgb]{0,0,0}\makebox(0,0)[lb]{\smash{$=$}}}%
    \put(0.77318342,0.27572572){\color[rgb]{0,0,0}\makebox(0,0)[lb]{\smash{\tiny $i$}}}%
    \put(0.84839681,0.27572572){\color[rgb]{0,0,0}\makebox(0,0)[lb]{\smash{\tiny $i+1$}}}%
    \put(0,0){\includegraphics[width=\unitlength,page=5]{AlgMult1.pdf}}%
    \put(0.17403566,0.12382372){\color[rgb]{0,0,0}\makebox(0,0)[lb]{\smash{$=$}}}%
    \put(0.29455244,0.0022223){\color[rgb]{0,0,0}\makebox(0,0)[lb]{\smash{\tiny $i$}}}%
    \put(0.36976588,0.0022223){\color[rgb]{0,0,0}\makebox(0,0)[lb]{\smash{\tiny $i+1$}}}%
    \put(0,0){\includegraphics[width=\unitlength,page=6]{AlgMult1.pdf}}%
    \put(0.63643171,0.0022223){\color[rgb]{0,0,0}\makebox(0,0)[lb]{\smash{\tiny $i$}}}%
    \put(0.71164515,0.0022223){\color[rgb]{0,0,0}\makebox(0,0)[lb]{\smash{\tiny $i+1$}}}%
    \put(0.81962342,0.12723808){\color[rgb]{0,0,0}\makebox(0,0)[lb]{\smash{$\cdot \hspace{1mm} ==U_{i} $}}}%
    \put(0.50223977,0.12382372){\color[rgb]{0,0,0}\makebox(0,0)[lb]{\smash{$=$}}}%
  \end{picture}%
\endgroup%

%% file: figs/X_i_down.pdf_tex
\begingroup%
  \makeatletter%
  \providecommand\color[2][]{%
    \errmessage{(Inkscape) Color is used for the text in Inkscape, but the package 'color.sty' is not loaded}%
    \renewcommand\color[2][]{}%
  }%
  \providecommand\transparent[1]{%
    \errmessage{(Inkscape) Transparency is used (non-zero) for the text in Inkscape, but the package 'transparent.sty' is not loaded}%
    \renewcommand\transparent[1]{}%
  }%
  \providecommand\rotatebox[2]{#2}%
  \ifx\svgwidth\undefined%
    \setlength{\unitlength}{554.47430516bp}%
    \ifx\svgscale\undefined%
      \relax%
    \else%
      \setlength{\unitlength}{\unitlength * \real{\svgscale}}%
    \fi%
  \else%
    \setlength{\unitlength}{\svgwidth}%
  \fi%
  \global\let\svgwidth\undefined%
  \global\let\svgscale\undefined%
  \makeatother%
  \begin{picture}(1,0.32717628)%
    \put(0.55966335,-0.02792596){\color[rgb]{0,0,0}\makebox(0,0)[lb]{\smash{}}}%
    \put(0,0){\includegraphics[width=\unitlength,page=1]{X_i_down.pdf}}%
    \put(0.48685411,0.00321004){\color[rgb]{0,0,0}\makebox(0,0)[lb]{\smash{$i$}}}%
    \put(0.63697211,0.00323936){\color[rgb]{0,0,0}\makebox(0,0)[lb]{\smash{$i+1$}}}%
    \put(0.78702415,0.00323936){\color[rgb]{0,0,0}\makebox(0,0)[lb]{\smash{$2n-1$}}}%
    \put(0.93280641,0.0031808){\color[rgb]{0,0,0}\makebox(0,0)[lb]{\smash{$2n$}}}%
    \put(0.48685411,0.31070401){\color[rgb]{0,0,0}\makebox(0,0)[lb]{\smash{$i$}}}%
    \put(0.63697211,0.31073334){\color[rgb]{0,0,0}\makebox(0,0)[lb]{\smash{$i+1$}}}%
    \put(0.78702415,0.31073334){\color[rgb]{0,0,0}\makebox(0,0)[lb]{\smash{$2n-1$}}}%
    \put(0.93280641,0.31067478){\color[rgb]{0,0,0}\makebox(0,0)[lb]{\smash{$2n$}}}%
    \put(0,0){\includegraphics[width=\unitlength,page=2]{X_i_down.pdf}}%
    \put(0.14635141,0.00321004){\color[rgb]{0,0,0}\makebox(0,0)[lb]{\smash{$1$}}}%
    \put(0.29646941,0.00323936){\color[rgb]{0,0,0}\makebox(0,0)[lb]{\smash{$2$}}}%
    \put(0.14635141,0.31070401){\color[rgb]{0,0,0}\makebox(0,0)[lb]{\smash{$1$}}}%
    \put(0.29646941,0.31073334){\color[rgb]{0,0,0}\makebox(0,0)[lb]{\smash{$2$}}}%
    \put(0,0){\includegraphics[width=\unitlength,page=3]{X_i_down.pdf}}%
    \put(0,0.14259644){\color[rgb]{0,0,0}\makebox(0,0)[lb]{\smash{\large $\mathsf{X}_{i}=$}}}%
    \put(-0.41537499,0.55214168){\color[rgb]{0,0,0}\makebox(0,0)[lb]{\smash{}}}%
    \put(0.49116281,0.0682662){\color[rgb]{0,0,0}\makebox(0,0)[lb]{\smash{$e_{1}$}}}%
    \put(0.65002554,0.07701094){\color[rgb]{0,0,0}\makebox(0,0)[lb]{\smash{}}}%
    \put(0.64864094,0.08707273){\color[rgb]{0,0,0}\makebox(0,0)[lb]{\smash{}}}%
    \put(0.64245749,0.0788281){\color[rgb]{0,0,0}\makebox(0,0)[lb]{\smash{}}}%
    \put(0.63832923,0.0682662){\color[rgb]{0,0,0}\makebox(0,0)[lb]{\smash{$e_{2}$}}}%
    \put(0.49116281,0.24428879){\color[rgb]{0,0,0}\makebox(0,0)[lb]{\smash{$e_{3}$}}}%
    \put(0.63832923,0.24428879){\color[rgb]{0,0,0}\makebox(0,0)[lb]{\smash{$e_{4}$}}}%
  \end{picture}%
\endgroup%